\documentclass[leqno,10pt]{amsart}
\usepackage{amssymb,amsmath,amsfonts,amsthm,mathtools,latexsym,ifthen,bezier,xspace,enumitem}
\usepackage{graphicx}

\newlength{\somenewlength}
\def\mathpalsave#1{\let\wasmathstyle=#1\relax}

\usepackage{amsxtra}
\usepackage{url}
\usepackage[english]{babel} 

\newcommand{\kla}[1]{ {\langle #1 \rangle} }
\newcommand{\anf}[1]{{\text{``}{#1}\text{''}}}

\newcommand{\st}{\;|\;}

\newcommand{\ran}{ {\rm ran} }
\newcommand{\otp}{ {\rm otp} }

\newcommand{\sub}{\subseteq}

\newfont{\ssi}{cmssi12 at 12pt}

\newcommand{\eins}{ {1{\rm\hspace{-0.5ex}l}} }

\newcommand{\cf}{ {\rm cf} }
\newcommand{\On}{ {\rm On} }
\newcommand{\Card}{{\rm Card}}
\newcommand{\leer}{\emptyset}
\newcommand{\ohne}{\setminus}

\newenvironment{ea*}{\begin{eqnarray*}}{\end{eqnarray*}}
\newcommand{\claim}[2]{
     \begin{enumerate}
       \item[{#1}] {\em #2}
     \end{enumerate}}

\setbox0=\hbox{$\longrightarrow$}
\newcommand{\To}{\longrightarrow}

\newcommand{\power}{{\mathcal{P}}}

\newcommand{\calC}{\mathcal{C}}

\newcommand{\calT}{\mathcal{T}}

\newcommand{\bI}{{\bar{I}}}

\newcommand{\bX}{{\bar{X}}}

\newcommand{\bgamma}{{\bar{\gamma}}}

\newcommand{\btau}{{\bar{\tau}}}

\newcommand{\btheta}{{\bar{\theta}}}

\newcommand{\bN}{{\bar{N}}}

\newcommand{\tA}{{\tilde{A}}}

\newcommand{\va}{{\vec{a}}}

\newcommand{\vx}{{\vec{x}}}

\newcommand{\seq}[2]{{\langle#1\;|\;}\linebreak[0]{#2\rangle}}

\renewcommand{\phi}{\varphi}
\newcommand{\card}[1]{\overline{\overline{#1}}}

\newcommand{\ZFC}{\ensuremath{\mathsf{ZFC}}\xspace}
\newcommand{\ZF}{\ensuremath{\mathsf{ZF}}}
\newcommand{\ZFCm}{\ensuremath{{\ZFC}^-}}

\newcommand{\SCH}{\ensuremath{\mathsf{SCH}}}
\newcommand{\V}{\ensuremath{\mathrm{V}}}

\newcommand{\forces}{\Vdash}

\def\<#1>{\langle#1\rangle}

\newcommand{\B}{{\mathord{\mathbb{B}}}}
\renewcommand{\P}{{\mathord{\mathbb P}}}
\newcommand{\Q}{{\mathord{\mathbb Q}}}

\newcommand{\N}{\mathord{\mathbb{N}}}

\newcommand{\MM}{\ensuremath{\mathsf{MM}}\xspace}

\newcommand{\PFA}{\ensuremath{\mathsf{PFA}}\xspace}

\newcommand{\SCFA}{\ensuremath{\mathsf{SCFA}}\xspace}

\newcommand{\BFA}{\ensuremath{\mathsf{BFA}}}
\newcommand{\BPFA}{\ensuremath{\mathsf{BPFA}}}
\newcommand{\BMM}{\ensuremath{\mathsf{BMM}}}

\newcommand{\MP}{\ensuremath{\mathsf{MP}}}
\newcommand{\lt}{{<}}
\newcommand{\closed}[1]{\ensuremath{\lt#1}\text{\rm-closed}}
\newcommand{\directed}[1]{\ensuremath{\lt#1}\text{\rm-directed-closed}}
\newcommand{\ColNothing}{\mathrm{Col}}
\newcommand{\Col}[1]{\ColNothing(#1)}

\newcommand{\MPSC}{\ensuremath{\MP_{\mathsf{SC}}}}

\newcommand{\MPColNothing}[1]{\MP_{\Col{\dot{\kappa}}}}

\newcommand{\bfMPsigmaclosed}{\MP_{\sigma\text{-closed}}(H_{\omega_2})}
\newcommand{\MPsigmaclosed}{\MP_{\sigma\text{-closed}}}
\newcommand{\bfMPSC}{\ensuremath{\MPSC(H_{\omega_2})}}

\newcommand{\isomorphic}{\cong}

\newcommand{\GCH}{\ensuremath{\mathsf{GCH}}\xspace}
\newcommand{\CH}{\ensuremath{\mathsf{CH}}\xspace}

\newcommand{\HOD}{\ensuremath{\mathsf{HOD}}}

\newcommand{\reals}{{\mathord{\mathbb{R}}}}

\newcommand{\Add}{\mathord{\mathrm{Add}}}

\newcommand{\Mantle}{\ensuremath{\mathbb{M}}}

\newcommand{\NS}{\mathsf{NS}}

\newcommand{\CC}{{\ensuremath{\text{$\sigma$-{\rm closed}}}}}

\newcommand{\BSCFA}{\ensuremath{\mathsf{BSCFA}}\xspace}

\newcommand{\Prikry}{P\v{r}\'{\i}kr\'{y}\xspace}
\newcommand{\Todorcevic}{Todor\v{c}evi\'{c}\xspace}
\newcommand{\Velickovic}{Veli\v{c}kovi\'{c}\xspace}

\newcommand{\HSC}{H_{\omega_2}}
\newcommand{\ccc}{\mathsf{ccc}}

\newtheorem{thm}{Theorem}[section]
\newtheorem*{thm*}{Theorem} 
\newtheorem{cor}[thm]{Corollary}
\newtheorem{lem}[thm]{Lemma}
\newtheorem{obs}[thm]{Observation}

\newtheorem{fact}[thm]{Fact}

\theoremstyle{definition}
\newtheorem{defn}[thm]{Definition}

\theoremstyle{remark}
\newtheorem{remark}[thm]{Remark}
\newtheorem{example}[thm]{Example}

\begin{document}

\title{Subcomplete forcing principles and definable well-orders}
\author{Gunter Fuchs}
\address{Department of Mathematics\\The College of Staten Island (CUNY)\\2800 Victory Boulevard\\Staten Island, NY 10314}
\address{and}
\address{The CUNY Graduate Center\\365 5th Avenue\\New York, NY 10016}
\email{gunter.fuchs@csi.cuny.edu}
\urladdr{www.math.csi.cuny.edu/~fuchs}
\thanks{The research for this paper was supported in part by PSC CUNY research grant 60630-00 48.}
\date{\today}
\maketitle

\newcommand{\bbC}{\mathbb{C}}
\newcommand{\Cn}{\ensuremath{C^{(n)}}}

\begin{abstract}
It is shown that the boldface maximality principle for subcomplete forcing, \bfMPSC, together with the assumption that the universe has only set-many grounds, implies the existence of a well-ordering of $\power(\omega_1)$ definable without parameters. The same conclusion follows from $\bfMPSC$, assuming there is no inner model with an inaccessible limit of measurable cardinals. Similarly, the bounded subcomplete forcing axiom, together with the assumption that $x^\#$ does not exist, for some $x\sub\omega$, implies the existence of a well-ordering of $\power(\omega_1)$ which is $\Delta_1$-definable without parameters, and $\Delta_1(H_{\omega_2})$-definable using a subset of $\omega_1$ as a parameter. This well-order is in $L(\power(\omega_1))$. Enhanced version of bounded forcing axioms are introduced that are strong enough to have the implications of $\bfMPSC$ mentioned above.
\end{abstract}

\section{Introduction}

This article is part of a larger project the theme of which is a comparison between the effects of subcomplete forcing principles and those for other, more familiar classes of forcing, such as proper, semiproper or stationary set preserving forcing notions.

Subcomplete forcing was introduced by Jensen \cite{Jensen:SPSCF}, see also \cite{Jensen2014:SubcompleteAndLForcingSingapore} for an overview article. Subcomplete forcing does not add reals, and the main result of \cite{Jensen:SPSCF} is that it can be iterated with revised countable support. It is pointed out in the same paper that every countably closed forcing is subcomplete. On the other hand, no nontrivial ccc forcing is subcomplete, see \cite{Kaethe:Diss}. Subcomplete forcing can change the cofinality of a regular cardinal to be countable; for example, assuming the continuum hypothesis, Namba forcing is subcomplete, and \Prikry{} forcing is subcomplete as well (these results can be found in Jensen \cite{Jensen2014:SubcompleteAndLForcingSingapore}), and the Magidor forcing to collapse the cofinality of a measurable cardinal of Mitchell order $\omega_1$ to $\omega_1$ is subcomplete (see Fuchs \cite{Fuchs:SubcompletenessOfMagidorForcing}). So there are subcomplete forcing notions that are not proper, and vice versa. Every subcomplete forcing preserves stationary subsets of $\omega_1$.

Jensen showed in \cite{Jensen:FAandCH} that one can force the subcomplete forcing axiom (\SCFA), that is, Martin's axiom for subcomplete forcing, over a model with a supercompact cardinal, in much the same way that one can force the proper forcing axiom, \PFA, under the same assumption. While \SCFA{} does not imply the continuum hypothesis, as Martin's Maximum implies $\SCFA+2^\omega=\omega_2$, the natural model resulting from a Baumgartner style iteration of subcomplete forcing notions will satisfy \SCFA, together with \CH, and even Jensen's combinatorial principle $\diamondsuit$, because subcomplete forcing does not add reals and preserves $\diamondsuit$. Being consistent with the continuum hypothesis makes \SCFA{} stand out.

It is now interesting to compare the consequences of \SCFA{} to those of \PFA{} and even \MM. Jensen showed that \SCFA{} implies \SCH{} and the failure of $\square_\tau$, for all uncountable cardinals $\tau$. In \cite{Fuchs:HierarchiesOfForcingAxioms}, I began a more detailed analysis of the effects of \SCFA{} (and its bounded versions) on the failure of \emph{weak} square principles, and in joint work with Rinot \cite{FuchsRinot:WeakSquareStationaryReflection}, this analysis was completed. It turned out that they are extremely close to the effects of \MM, and the only difference seems to be attributable to the fact that \SCFA{} is consistent with \CH, while \MM is not. In the paper \cite{Fuchs:DiagonalReflection}, I determined the effects of \SCFA{} on the failure of weak variants of the \Todorcevic{} square principles, and again, the situation turned out to be very similar to that of \MM{} - the only difference stemming from the fact that \SCFA{} does not imply the failure of $\CH$. Other previous research focused on forcing principles for subcomplete forcing other than the traditional axioms, such as resurrection axioms (see \cite{Fuchs:HierarchiesOfRA}) or maximality principles (see \cite{Kaethe:Diss}).

In the present paper, I explore another instance of the somewhat surprising phenomenon that despite the substantial difference between proper and subcomplete forcing, their forcing principles nevertheless have very similar effects on the set theoretic universe.

Namely, I investigate situations in which forcing principles for subcomplete forcing imply the existence of definable well-orderings of $\power(\omega_1)$. There is a history of results on the existence of definable well-orderings of $\reals$, and even of $\power(\omega_1)$, as a consequence of forcing axioms. Results include (chronologically): $\MM\implies 2^\omega=\omega_2$ (Foreman-Magidor-Shelah, \cite{FMS:MM1}), $\PFA\implies 2^\omega=\omega_2$ (\Velickovic{} \cite{Velickovic:ForcingAxiomsAndStationarySets}, \Todorcevic{} \cite{Bekkali:TopicsInST}), $\MM\implies$ there is a well-ordering of $\reals$ definable (with parameters) in $\kla{H_{\omega_2},\in}$ (Woodin \cite{Woodin:ADforcingAxiomsNonstationaryIdeal}), $\BMM$ implies there is a well-ordering of $\reals$ definable (with parameters) in $\kla{H_{\omega_2},\in}$ (\Todorcevic \cite{Todorcevic:GenericAbsolutenessAndContinuum}), $\BPFA$ implies there is a well-ordering of $\power(\omega_1)$ $\Delta_2$-definable (with parameters) in $\kla{H_{\omega_2},\in}$ (Moore \cite{Moore:SMR}),
$\BPFA$ implies there is a well-ordering of $\power(\omega_1)$ $\Delta_1$-definable (with parameters) in $\kla{H_{\omega_2},\in}$ (Caicedo-Velickovic \cite{CaicedoVelickovic:BPFAandWOofR}).

By completely different methods, I will show that certain subcomplete forcing principles have similar effects on the existence of definable well-orders of $\power(\omega_1)$, under appropriate additional assumptions.
In Section \ref{sec:MPSC}, the forcing principle under consideration is the boldface maximality principle for subcomplete forcing, $\bfMPSC$. This is the scheme expressing that every statement about an element of $H_{\omega_2}$ that can be forced to be true by a subcomplete forcing in such a way that it will remain in every further forcing extension by a subcomplete forcing is already true in $\V$. The additional assumption I use in this section is that the universe has only set-many inner models (grounds) of which it is a set-forcing extension. This is maybe an unexpected appearance of an assumption on the set-theoretic geology (see \cite{FuchsHamkinsReitz:Geology}) of the ambient universe, but it guarantees that the mantle $\Mantle$, the intersection of all grounds, is itself a ground model of the universe, hence is very ``close to $\V$''. It is known that the mantle is invariant under set forcing, and hence, the assumption of set many grounds provides us with a forcing invariant inner model that's close to $\V$. By work of Usuba (\cite{Usuba:DDGandVeryLargeCardinals}), the assumption that there are only set many grounds follows from the existence of a rather strong large cardinal, called hyper huge. I recap the basics of set-theoretic geology and maximality principles in more detail in this section, and then prove the main result, that if \MPSC{} holds and there are only set-many grounds, then there is a well-order of $\power(\omega_1)$, of order-type $\omega_2$, definable \emph{without parameters.} This is Theorem \ref{thm:WOofPomega1FrombfMPSCandSetManyGrounds}. A corollary of this theorem is that if $\bfMPSC$ holds and there are only set many grounds, and if a forcing notion $\P$ preserves $\bfMPSC$ and $\omega_2$, then $\P$ cannot add subsets of $\omega_1$. This is Corollary \ref{cor:PreservingbfMPSC}. The same conclusions can be made assuming \bfMPSC{} and the absence of an inner model with an inaccessible limit of measurable cardinals, see Corollary \ref{cor:bfMPSC+NoIMwithAnIAlimitOfMeasurablesImpliesWO}.

It is easy to see that $\bfMPSC$ is a a strengthening of the version of the bounded forcing axiom for subcomplete forcing, \BSCFA. In Section \ref{sec:BSCFA}, I deal with this latter principle, and show that in the absence of $0^\#$, or just of $x^\#$, for some $x\sub\omega_1$, \BSCFA{} implies the existence of a well-order of $\power(\omega_1)$ of order type $\omega_2$, definable in $\kla{H_{\omega_2},\in}$ in a $\Delta_1$ way, using a subset of $\omega_1$ as a parameter, and this well-order is in $L(\power(\omega_1))$ (which thus is a model of $\ZFC$). This is Lemma \ref{lem:BSCFA+NoZeroSharpImpliesWOofOmega1}. A similar conclusion on the preservation of \BSCFA{} (in the absence of $0^\#$) under forcing is made in Lemma \ref{lem:NonPreservationOfBSCFA}: if $\P$ preserves $\omega_2$ and $\BSCFA$, then it cannot add a subset of $\omega_1$.

In Section \ref{sec:MoreReflection}, I make an excursion on ways to strengthen the bounded forcing axiom for different forcing classes. The motivation for doing this is that I want to find principles located between \BSCFA and $\bfMPSC$ that are still strong enough to yield the results of Section \ref{sec:MPSC}. The results in this section are not needed for the following section, but are of independent interest. I argue that the ``correct'' version of the bounded forcing axiom for countably closed forcing notions should be the statement that whenever $G$ is generic for a countably closed forcing notion, then $H_{\omega_2}\prec_{\Sigma_2}H_{\omega_2}^{\V[G]}$. The reason is that this statement has the same consistency strength as the bounded forcing axiom for other iterable forcing classes, such as the collection of proper, semi-proper or subcomplete forcing notions, and that it makes an analogous statement about generic absoluteness as the characterization of the traditional bounded forcing axioms by Bagaria: the level of elementarity is one more than guaranteed by \ZFC.

Section \ref{sec:EBSCFA} introduces enhanced bounded forcing axioms, essentially guaranteeing $\Sigma_1$-elementarity with respect to the structure $\kla{\HSC,\in,I\cap\HSC}$, where $I$ is some adequate definable class. I calculate the consistency strength of an enhanced principle that is strong enough to guarantee the conclusions of Section \ref{sec:MPSC}, and develop a type of large cardinal that allows us to produce forcing extensions where this principle holds.


\section{Well-orders from maximality principles for subcomplete forcing}
\label{sec:MPSC}

The first forcing principle I will look at is a maximality principle, introduced in generality by Stavi and V\"{a}\"{a}n\"{a}nen \cite{StaviVaananen:ReflectionPrinciples} and Hamkins \cite{Hamkins2003:MaximalityPrinciple}.

\begin{defn}%
\label{defn:MP}
Let $\Gamma$ be a class of forcing notions (here we can take it to be definable without parameters), and let $X$ be a term defining a set (again, for the present purposes, it can be taken to be a parameter free definition). Then $\MP_\Gamma(X)$, the \emph{maximality principle for $\Gamma$, with parameters from $X$,} is the scheme of formulas asserting, for every formula $\phi(\vx)$ and for all $\va\in X$, that if $\phi(\va)$ can be forced to be true by a forcing notion $\P\in\Gamma$ in such a way that for every $\P$-name $\dot{\Q}$ such that $\forces_\P$``$\dot{\Q}\in\Gamma$'', $\forces_{\P*\dot{\Q}}\phi(\va)$, then $\phi(\va)$ holds already in $\V$.

If $\Gamma$ is the class of subcomplete forcing notions, then I write $\MPSC(X)$ for the principle. The boldface maximality principle for subcomplete forcing is when $X=H_{\omega_2}$, so $\bfMPSC$. If $\Gamma$ is the class of countably closed forcing notions, the resulting principle is denoted $\MPsigmaclosed(X)$, and again, the boldface maximality principle for countably closed forcing uses the parameter set $X=H_{\omega_2}$, denoted $\bfMPsigmaclosed$.
\end{defn}

In the context of maximality principles for subcomplete forcing, I will say (as is customary) that a statement $\phi$ is \emph{subcomplete-forceable} if it can be forced to hold by a subcomplete forcing, and it is \emph{subcomplete-necessary} if it holds in any forcing extension obtained by subcomplete forcing. It is subcomplete-forceably necessary if the statement ``$\phi$ is subcomplete-necessary'' is subcomplete-forceable.

The maximality principles for countably closed forcing notions (as well as ${<}\kappa$-closed forcing notions, and other classes) were studied in detail in \cite{Fuchs:MPclosed}. The versions for subcomplete forcing were considered in \cite{Kaethe:Diss}, and the emerging picture was that the boldface maximality principles for these two classes have very similar consequences. For example, they both imply Jensen's $\diamondsuit$ principle and the nonexistence of Kurepa trees. The present work will indicate that they are rather different, after all, at least under a suitable assumption on the set-theoretic geology of the universe under consideration.


To explain this assumption, I will have to say a few words about set-theoretic geology. Research in this area, initiated in \cite{FuchsHamkinsReitz:Geology}, is concerned with the structure of \emph{grounds} of a universe $\V$. An inner model $M$ of $\ZFC$ is a ground if $\V=M[g]$, for some $g$ which is generic over $M$ for some forcing $\P\in M$. The chief object of study is the mantle $\Mantle$, that is the intersection of all grounds. It was shown in \cite[Corollary 13]{FuchsHamkinsReitz:Geology} that the mantle $\Mantle$ is a definable class, and the following strong downward directedness of grounds hypothesis (strong DDG) was isolated there (\cite[Definition 19]{FuchsHamkinsReitz:Geology}): the grounds of the universe are downward set-directed. It was shown in \cite[Theorem 22]{FuchsHamkinsReitz:Geology} that the strong DDG implies that the mantle \Mantle{} is a model of \ZFC. An auxiliary inner model, called the generic mantle, was also introduced in \cite[Definition 42]{FuchsHamkinsReitz:Geology}, defined to be the intersection of all mantles of all set-forcing extensions. It was shown in \cite[Theorem 44]{FuchsHamkinsReitz:Geology} that the generic mantle is an inner model of \ZF{} invariant under set-forcing. Finally, in \cite[Corollary 51]{FuchsHamkinsReitz:Geology}, it was shown that if the generic DDG holds, that is, if in all set-forcing extensions, the grounds are downward directed, then the mantle and the generic mantle coincide.

In a major step for set-theoretic geology, it was shown recently in \cite[Theorem 1.3]{Usuba:DDGandVeryLargeCardinals} that it is a \ZFC{} theorem that the grounds are downward set directed. In particular, the generic DDG also holds, and, putting this together with the previously known implications of downward directedness, it follows that the mantle \Mantle{} is a forcing-invariant model of \ZFC. This fact also shows that if there are only set-many grounds, then there is a smallest ground, known as a bedrock (see \cite{Reitz:TheGroundAxiom}), which in this situation is equal to the mantle. So if there are only set-many grounds, then the universe is a set-forcing extension of its mantle, and it is easy to see that these conditions are equivalent, because if $\V=\Mantle[G]$, where $G$ is generic over $\Mantle$ for some complete Boolean algebra $\B\in\Mantle$, then every ground $W$ of $\V$ is squeezed in between $\Mantle$ and $\V$, and hence is a forcing extension of $\Mantle$ by a subalgebra of $\B$, so there are only set-many possibilities for $W$. It is this assumption on geology that is used in the theorem on the existence of a definable well-ordering of $\power(\omega_1)$.

I would now like to give some background on subcomplete forcing. The concept was introduced by Jensen \cite{Jensen:SPSCF}.

\begin{defn} A transitive set $N$ (usually a model of $\ZFCm$) is \emph{full} if there is an ordinal $\gamma$ such that $L_\gamma(N)\models\ZFCm$ and $N$ is regular in $L_\gamma(N)$, meaning that if $x\in N$, $f\in L_\gamma(N)$ and $f:x\longrightarrow N$, then $\ran(f)\in N$.
\end{defn}

\begin{defn} For a poset $\P$, $\delta(\P)$ is the minimal cardinality of a dense subset of $\P$.
\end{defn}

\begin{defn} Let $N=L^A_\tau=\kla{L_\tau[A],\in,A\cap L_\tau[A]}$ be a $\ZFCm$ model, $\varepsilon$ an ordinal and $X\cup\{\varepsilon\}\sub N$. Then $C^N_\varepsilon(X)$ is the smallest $Y\prec N$ (with respect to inclusion) such that $X\cup\varepsilon\sub Y$.
\end{defn}

Models $N$ of the form described in the previous definition have definable Skolem-functions, so that the definition of $C^N_\varepsilon(X)$ makes sense.
\begin{defn} A forcing notion $\P$ is \emph{subcomplete} if there is a cardinal $\theta$ which verifies the subcompleteness of $\P$, which means that $\P\in H_\theta$, and for any $\ZFCm$ model $N=L_\tau^A$ with $\theta<\tau$ and $H_\theta\sub L_\tau[A]$, any $\sigma:\bN\prec N$ such that $\bN$ is countable, transitive and full and such that $\P,\theta\in\ran(\sigma)$, any $\bar{G}\sub\bar{\P}$ which is $\bar{\P}$-generic over $\bN$, and any $s\in\ran(\sigma)$, the following holds. Letting $\sigma(\kla{\bar{s},\bar{\theta},\bar{\P}})=\kla{s,\theta,\P}$, there is a condition $p\in\P$ such that whenever $G\sub\P$ is $\P$-generic over $\V$ with $p\in G$, there is in $\V[G]$ a $\sigma'$ such that
\begin{enumerate}
  \item $\sigma':\bN\prec N$,
  \item $\sigma'(\kla{\bar{s},\bar{\theta},\bar{\P}})=\kla{s,\theta,\P}$,
  \item $(\sigma')``\bar{G}\sub G$,
  \item $C^N_{\delta(\P)}(\ran(\sigma'))=C^N_{\delta(\P)}(\ran(\sigma))$.
\end{enumerate}
\end{defn}

Jensen showed that if one requires $\sigma'=\sigma$ in the previous definition, the resulting concept is equivalent to saying that the complete Boolean algebra of $\P$ is isomorphic to that of a countably closed forcing. Thus, in a sense, subcompleteness can be viewed as a natural weakening of countable closure. The main fact on subcomplete forcing that I will need is the following remarkable theorem of Jensen \cite{Jensen:ExtendedNamba}.

\begin{thm}
\label{thm:ExtendedNambaForcing}
Let $\kappa$ be an inaccessible cardinal, and assume that \GCH{} holds below $\kappa$. Let $A\sub\kappa$ be a set of regular cardinals. Then there is a subcomplete, $\kappa$-c.c.~forcing $\P$ of size $\kappa$ such that if $G$ is $\P$-generic, then $\kappa=\omega_2^{\V[G]}$ and for every regular $\tau\in(\omega_1,\kappa)$,
\[\cf^{V[G]}(\tau)=\left\{
\begin{array}{l@{\qquad}l}
\omega_1 &\text{if $\tau\in A$}\\
\omega &\text{otherwise.}
\end{array}
\right.\]
\end{thm}

I will call this forcing the \emph{extended Namba forcing for $A$,} and denote it by $\N_{A,\kappa}$. The idea is to use it to code subsets of $\omega_1$ into the ``cofinality $\omega$/$\omega_1$ pattern''.

\begin{thm}
\label{thm:WOofPomega1FrombfMPSCandSetManyGrounds}
Assume $\bfMPSC$, and assume that there are only set-many grounds. Then there is a well-ordering of $\power(\omega_1)$, definable without parameters, of order type $\omega_2$ (and in particular, $2^{\omega_1}=\omega_2$).
\end{thm}

\begin{proof}
If there are only set-many grounds, then all of these grounds have a common ground, by Usuba's result. This common ground then is the mantle $\Mantle$, and it follows that $\V=\Mantle[g]$, for some $g$ which is generic over $\Mantle$ for some forcing $\P$ in $\Mantle$. Since $\Mantle$ is forcing-absolute, letting $\delta=\omega_2$, it follows from $\bfMPSC$ that $\V_\delta^\Mantle\prec\Mantle$, by verifying the Tarski-Vaught criterion for elementary substructures; see the proof of \cite[Theorem 3.8]{Fuchs:MPclosed} and \cite[Lemma 4.15]{Fuchs:MPclosed}.
Since $\delta$ is regular, it follows that $\delta$ is inaccessible in $\Mantle$. This, in turn, implies that there is a proper class of inaccessible cardinals in $\Mantle$, and since $\V$ is a set forcing extension of $\Mantle$, also in $\V$.

Let $\alpha\ge\card{\P}^+$, and let $\seq{\kappa_i}{i\le\omega_1}$ enumerate the next $\omega_1+1$ inaccessible cardinals above $\alpha$.
We can perform an Easton iteration of at least countably closed forcing notions in order to reach an extension in which $\GCH$ holds below $\kappa_{\omega_1}$, such that if $h$ is generic, then each $\kappa_i$ is still inaccessible in $\V[h]$. Now, given a subset $A$ of $\omega_1$ in $\V$, let $\tA=\{\kappa_i\st i\in A\}$. Then let $\Q=\N_{\tA,\kappa_{\omega_1}}$ be the extended Namba forcing for $\tA$ in $\V[h]$. Thus, $\Q$ is $\kappa_{\omega_1}$-c.c., and if $G$ is $\Q$-generic over $\V[h]$, then for every $\mu\in\tA$, $\V[h][G]$ thinks that $\cf(\mu)=\omega_1$, and for every $\nu\in\kappa_{\omega_1}\ohne\tA$ which is regular in $\V[h]$, $\V[h][G]$ thinks that $\cf(\nu)=\omega$. In particular, the latter is true for every $\kappa_j$ with $j<\omega_1$, $j\notin A$, since $\kappa_j$ remains regular in $\V[h]$. $\kappa_{\omega_1}$ becomes $\omega_2$ in $\V[h][G]$.

Now if $\V[h][G][I]$ is a further subcomplete forcing extension of $\V[h][G]$, then the cofinality of $\kappa_i$, for $i<\omega_1$, cannot change, since subcomplete forcing does not add reals. Moreover, the sequence $\seq{\kappa_i}{i<\omega_1}$ is definable in $\V[h][G][I]$ from the parameter $\alpha$, as the enumeration of the next $\omega_1$ many inaccessible cardinals in $\Mantle$ beyond $\alpha$. Hence, $A$ is definable in $\V[h][G][I]$ as the set of $i<\omega_1$ such that $\cf(\kappa_i)=\omega_1$. Let $\psi(A,\alpha)$ be the statement expressing that if $\seq{\lambda_i}{i<\omega_1}$ enumerates the next $\omega_1$ many inaccessible cardinals of $\Mantle$ beyond $\alpha$, then for all $i<\omega_1$, $i\in A$ iff $\cf(\lambda_i)=\omega_1$. Then the statement
$\phi(A)$, expressing that there is an $\alpha$ such that $\psi(A,\alpha)$ holds in $\V[h][G][I]$. Since $I$ was generic for an arbitrary subcomplete forcing notion in $\V[h][G]$, this means that $\phi(A)$ is necessary with respect to subcomplete forcing extensions in $\V[h][G]$, and hence it is forceably necessary with respect to subcomplete forcing in $\V$. It follows by $\bfMPSC$ that $\phi(A)$ already holds in $\V$. Since $\phi(A)$ was an arbitrary subset of $\omega_1$ in $\V$, it follows that $\phi(B)$ holds, for every $B\sub\omega_1$.

Given $A\sub\omega_1$, let $\beta$ be such that $\psi(A,\beta)$ holds. In $\V^{\Col{\omega_1,\beta}}$, and in any further subcomplete forcing extension, $\psi(A,\beta)$ continues to hold, and $\beta<\omega_2$ there. Thus, if we let $\alpha_A$ be least such that $\psi(A,\alpha_A)$ holds, it is subcomplete forceably necessary that $\alpha_A<\omega_2$, and so, it is already less than $\omega_2$. This shows that if we define $A<^*B$ iff $\alpha_A<\alpha_B$, for $A,B\sub\power(\omega_1)$, then this is a well-ordering of $\power(\omega_1)$ of order type $\omega_2$.
\end{proof}

Note that the conclusion of the previous theorem implies that $\power(\omega_1)\sub\HOD$.

Of course, we don't really need the coding points to be inaccessible cardinals. The following cardinals will do.

\begin{defn}
\label{defn:GCHsurvivor}
A regular cardinal $\kappa\ge\omega_2$ is \emph{a GCH survivor} if it remains a regular cardinal after performing the standard forcing to force \GCH.
\end{defn}

For example, successors of strong limit cardinals are GCH survivors. All we needed in the proof of the previous theorem was that inaccessible cardinals are GCH survivors, and since we don't need a proper class of inaccessible cardinals, we will be able to work with the lightface maximality principle $\MPSC$, that is $\MPSC(\leer)$, in the following theorem, but we will only get a well-order of $\reals$. Note that the lightface $\MPSC$ already implies $\CH$. This is because the standard forcing $\Add(\omega_1,1)$ to add a new subset to $\omega_1$ with countable conditions forces \CH, and it is countably closed, hence subcomplete. Further subcomplete forcing cannot add reals, and hence preserves \CH. Thus, \CH{} is subcomplete-forceably necessary.

\begin{thm}
\label{thm:MPSCgiveWOofR}
Assume $\MPSC$, and assume that there are only set-many grounds. Then there is a well-ordering of $\reals$, definable without parameters.
\end{thm}

\begin{proof}
Let $\alpha$ be greater than the chain condition of the forcing leading from $\Mantle$ to $\V$, and let $\seq{\kappa_n}{n<\omega}$ enumerate the next $\omega$ many GCH survivors above $\alpha$.
We can perform an Easton iteration of at least countably closed forcing notions in order to reach an extension in which $\GCH$ holds below $\tilde{\kappa}=\sup_{n<\omega}\kappa_n$, such that if $h$ is generic, each $\kappa_n$ is still a regular cardinal $\V[h]$. Now, given a real $a\sub\omega$ in $\V$, and letting $\tilde{a}=\{\kappa_n\st n\in a\}$, the extended Namba forcing $\N_{\tilde{a},\tilde{\kappa}}$ can be used to reach a model $\V[h][G]$ that thinks that for every $n<\omega$, $\cf(\kappa_n)=\omega_1$ iff $n\in a$ and $\cf(\kappa_n)=\omega$ iff $n\notin a$, even though we are not working below an inaccessible cardinal, see \cite[Chapter 3, pp.~16-17]{Jensen:FAandCH}; $\tilde{\kappa}$ of course has countable cofinality already in $\V$.

Now if $\V[h][G][I]$ is a further subcomplete forcing extension of $\V[h][G]$, then as before, the cofinality of $\kappa_n$, for $n<\omega$, cannot change. Moreover, the sequence $\seq{\kappa_n}{n<\omega}$ is definable in $\V[h][G][I]$ from the parameter $\alpha$, as the enumeration of the next $\omega$ many ordinals beyond $\alpha$ that are GCH survivors in $\Mantle$. Hence, $a$ is definable as the set of $n<\omega$ such that $\cf(\kappa_n)=\omega_1$. So the statement $\phi(a)$, expressing that there is an $\alpha$ such that $\psi(a,\alpha)$ holds, meaning that if $\seq{\lambda_i}{i<\omega_1}$ enumerates the next $\omega$ many GCH survivors beyond $\alpha$ in $\Mantle$, then for all $n<\omega$, $n\in a$ iff $\cf(\lambda_n)=\omega_1$, is subcomplete-forceably necessary, and hence true in $\V$, by $\MPSC$. Note that the lightface $\MPSC$ implies the form that allows parameters from $H_{\omega_1}$, see \cite[Theorem 2.6]{Fuchs:MPclosed}, \cite[Lemma 1.10, Theorem 1.11, Corollary 1.12]{Leibman:Diss}; the point is that subcomplete forcing does not change $H_{\omega_1}$.

Now, the proof can be completed as before. Given $a\sub\omega$, let $\beta$ be such that $\psi(a,\beta)$ holds. In $\V^{\Col{\omega_1,\beta}}$, and in any further subcomplete forcing extension, $\psi(a,\beta)$ continues to hold, and $\beta<\omega_2$ there. Thus, if we let $\alpha_a$ be least such that $\psi(a,\alpha_a)$ holds, it is subcomplete forceably necessary that $\alpha_a<\omega_2$, and so, it is already less than $\omega_2$. This shows that if we define $a<b$ iff $\alpha_a<\alpha_b$, for $a,b\sub\omega$, then this is a well-ordering of $\power(\omega)$, as wished.
\end{proof}

As before, the conclusion of the previous theorem implies that $\power(\omega)\sub\HOD$.

It is easy to see that $\bfMPSC$ is preserved by subcomplete forcing notions that don't change $H_{\omega_2}$, that is, that don't add subsets of $\omega_1$ (see \cite[Lemma 4.2]{Fuchs:MPclosed} for the corresponding fact in the context of ${<}\kappa$-closed forcing, the proof of which easily generalizes). In particular, it is preserved by ${<}\omega_2$-distributive forcing notions, since by \cite[Theorem 6.2 and Observation 2.4]{Fuchs:ParametricSubcompleteness}, every ${<}\omega_2$-distributive forcing notion is subcomplete. The proof of Theorem \ref{thm:WOofPomega1FrombfMPSCandSetManyGrounds} shows that the requirement of not adding subsets of $\omega_1$ is needed in order to be able to conclude that a subcomplete forcing preserve $\bfMPSC$, if one insists that it preserve $\omega_2$.

\begin{cor}
\label{cor:PreservingbfMPSC}
Assume there are only set-many grounds.
Suppose $N$ is a set-forcing extension of $\V$ such that $\bfMPSC$ holds in $N$, and that $\omega_2^\V=\omega_2^N$. Then $\power(\omega_1)^\V=\power(\omega_1)^N$.
\end{cor}

\begin{proof}
Note that $\omega_1^\V=\omega_1^N$, so I won't distinguish between the two.
Given a set $A\sub\omega_1$ in $N$, working in $N$, it follows from the proof of Theorem \ref{thm:WOofPomega1FrombfMPSCandSetManyGrounds} that there is some $\alpha<\omega_2$ such that if one lets $\seq{\kappa_i}{i<\omega_1}$ enumerate the next $\omega_1$ many inaccessible cardinals of $\Mantle^N$ beyond $\alpha$, then for all $i<\omega_1$, $i\in A$ iff $\cf(\kappa_i)=\omega_1$, and $i\notin A$ iff $\cf(\kappa_i)=\omega$. This is because $N$ also only has set-many grounds. This is because we know that the mantle $\Mantle$ of $\V$ is the same as the mantle of $N$, and that $\Mantle$ is a ground of $\V$, which is a ground of $N$. So the mantle of $N$ is a ground of $N$, and it was pointed out in the discussion after Definition \ref{defn:MP} that this is equivalent to saying that $N$ has only set many grounds.

Moreover, each $\kappa_i$ is less than $\omega_2$. But since $\omega_2^\V=\omega_2^N$, it follows that $\cf(\kappa_i)^\V=\cf(\kappa_i)^N$: $\cf^\V(\kappa_i)$ can only be $\omega$ or $\omega_1$, since $\kappa_i<\omega_2$. If $\cf(\kappa_i)^N=\omega_1$, then this is true in $\V$ as well, because $\V\sub N$. And if $\cf(\kappa_i)^N=\omega$, then this holds in $\V$, or else $\omega_1^\V$ would have to be countable in $N$.

Now, since $\Mantle^\V=\Mantle^N$, it follows that $\vec{\kappa}$ is definable in $\V$ from $\alpha$, and $A$ is definable from $\vec{\kappa}$, so $A\in\V$.
\end{proof}

Focusing on subcomplete forcing notions, we get the following characterization.

\begin{cor}
Assume $\bfMPSC$ and there are only set-many grounds, and let $\P$ be a subcomplete forcing that preserves $\omega_2$. Then $\P$ preserves $\bfMPSC$ iff $\P$ doesn't add subsets of $\omega_1$.
\end{cor}

These past two corollaries are in stark contrast to the situation with the boldface maximality principle for countably closed forcing, $\bfMPsigmaclosed$ (see \cite[Lemma 4.10]{Fuchs:MPclosed}):

\begin{fact}
$\bfMPsigmaclosed$ is preserved by the forcing $\Add(\omega_1,\kappa)$, for any $\kappa$.
\end{fact}

The requirement of preserving $\omega_2$ is necessary in the previous two corollaries, because if there is a fully reflecting cardinal $\kappa$ in the original model of $\bfMPSC$, then one can force to another model of $\bfMPSC$ by doing the standard subcomplete forcing iteration, which will render $\kappa$ the new $\omega_2$, and thus will add subsets to $\omega_1$.

To put the assumption that there are only set-many grounds in context, I would like to point out the following fact, due to Usuba, see \cite[Theorem 1.6]{Usuba:DDGandVeryLargeCardinals}.

\begin{fact}[Usuba]
If there is a hyper-huge cardinal $\kappa$, then there are less than $\kappa$ many grounds.
\end{fact}

Here $\kappa$ is hyper-huge if for every $\lambda$, there is a $j:\V\To M$ such that $j(\kappa)>\lambda$ and ${}^{j(\lambda)}M\sub M$. So the assumption of the previous theorems follows from the existence of large cardinals. It will turn out that one can get a similar conclusion from anti-large-cardinal hypotheses. The following is a technical lemma that tries to get by with as weak an assumption as possible. The status of this assumption is unclear. It certainly holds if there are only set-many grounds. Let's introduce the following notation.

\begin{defn}
For a set $X$ of ordinals and $i<\otp(X)$, let $(X)_i$ be the $i$-th element of $X$ in its monotone enumeration.
\end{defn}

\begin{lem}
\label{lem:WOfromTechnicalAssumption}
Assume $\bfMPSC$, and assume that $\psi(x,y,z)$ is a formula in the language of set theory such that for some set $P\sub\omega_1$, we have that for some $\alpha\in\On$,
\[I_{\alpha,P}=\{\beta\in\On\st\chi(\beta,\alpha,P)\}\]
has size at least $\omega_1$ and consists of GCH survivors.
Suppose, moreover, that for every subcomplete forcing $\P$, if $G$ is generic for $\P$ over $\V$, then in $\V[G]$, there is an ordinal $\alpha'$ such that
\[I_{\alpha,P}=(I_{\alpha',P})^{V[G]}\]
Then there is a well-order of $\power(\omega_1)$, definable from $P$.
\end{lem}

\begin{proof}
Given a set $A\sub\omega_1$, and fixing $P$ and $\alpha$ as above, let $\psi(\alpha,A,P)$ be the statement saying ``for all $i<\omega_1$, $i\in A$ iff $\cf((I_{\alpha,P})_i)=\omega_1$ and $i\notin A$ iff $\cf((I_{\alpha,P})_i)=\omega$''. The statement $\phi(A,P)$, saying ``there is an $\alpha'$ such that $\psi(\alpha',A,P)$ holds'' is then subcomplete-forceably necessary, and hence true. So we can let $\alpha_A$ be the least $\alpha'$ such that $\psi(\alpha',A,P)$ holds and say that $A<B$ iff $\alpha_A<\alpha_B$.
\end{proof}

Note that if in the previous lemma, $\chi(x,y,z)$ is absolute with respect to subcomplete forcing, then it must be the case that for every $\gamma$, there is an $\alpha>\gamma$ such that $I_{\alpha,P}$ has order type at least $\omega_1$, consists of GCH survivors, and has $\min(I_{\alpha,P})>\gamma$. Because otherwise, if $\gamma$ is a counterexample, then in $\V^{\Col{\omega_1,\gamma}}$, it's subcomplete-necessary that there is no $\alpha$ such that $I_{\alpha,P}$ is as described, and so there is no such $\alpha$ in $\V$.

\begin{cor}
\label{cor:bfMPSC+NoIMwithAnIAlimitOfMeasurablesImpliesWO}
Assume $\bfMPSC$, and assume there is no inner model with an inaccessible limit of measurable cardinals. Then there is a well-order of $\power(\omega_1)$ of order type $\omega_2$, definable from a subset of $\omega_1$.
\end{cor}

\begin{proof}
Under the assumption, the core model $K$ exists. Letting $\delta=\omega_2$, since $K$ is forcing invariant, it follows that $K|\delta\prec K$, as before. This implies that the class of measurable cardinals in $K$ is bounded, as otherwise, $\delta$ would be an inaccessible limit of measurable cardinals in $K$. Thus, the set of measurable cardinals of $K$ is bounded in $\delta$. By Mitchell's covering lemma of \cite{Mitchell:OnTheSCH}, there is a ``maximal'' sequence of indiscernibles $\calC$, which can be viewed as a bounded subset of $\delta$, and hence can be coded in a simple way by a set $D\sub\omega_1$, such that for every uncountable set $X$, there is a set $Y\in K[D]$ of the same cardinality as $X$, with $X\sub Y$.

Let $\seq{\kappa_i}{i<\omega_1}$ enumerate the first $\omega_1$ many GCH-survivors greater than $\omega_1$ in $\V$, and let $I=\{\kappa_i\st i<\omega_1\}$. Let $Z\in K[D]$ be the $<_{K[D]}$-least set of ordinals such that $I\sub Z$ and $Z$ has cardinality $\omega_1$. Let $\bar{I}$ be the image of $I$ under the Mostowski-collapse of $Z$. $\bar{I}$ is a bounded subset of $\omega_2$.

Clearly then, $I\in K[D][\bar{I}]$. Let $D^*$ code $D$ and $\bar{I}$ as a subset of $\omega_1$, so that $I\in K[D^*]$. Let $I$ be the $\alpha$-th element of $K[D^*]$ in the canonical well-order of $K[D^*]$, and let $\psi(x,y,z)$ be the formula expressing that $y$ is an ordinal, and that $x$ belongs to the $y$-th set of $K[z]$. If $I_{\alpha,P}$ is defined from $\psi$ as in the statement of Lemma \ref{lem:WOfromTechnicalAssumption}, then we get that $I_{\alpha,D^*}$ has size $\omega_1$ and consists of GCH-survivors. By the forcing absoluteness of $K[D^*]$ and Lemma \ref{lem:WOfromTechnicalAssumption} then, there is a well-order of $\power(\omega_1)$ of order type $\omega_2$, definable from a subset of $\omega_1$.
\end{proof}

%

\section{Well-orders from the bounded subcomplete forcing axiom}
\label{sec:BSCFA}

\begin{defn}[\cite{GoldsternShelah:BPFA}]
\label{defn:BFA}
Let $\P$ be a notion of forcing, and let $\B$ be its Boolean completion. Then the bounded forcing axiom for $\P$ says that given any collection of $\omega_1$ many maximal antichains in $\B$, each having size at most $\omega_1$, there is a filter in $\B$ that meets each antichain in the collection. If $\Gamma$ is a class of forcing notions, then the bounded forcing axiom for $\Gamma$, denoted $\BFA_\Gamma$, says that the bounded forcing axiom holds for every $\P\in\Gamma$.
\end{defn}

Bagaria \cite{Bagaria:BFAasPrinciplesOfGenAbsoluteness} showed that the bounded forcing axiom can be expressed as a principle of generic absoluteness, as follows.

\begin{lem}
\label{lem:BagariaForIndividualPoset}
The bounded forcing axiom for a poset $\P$ is equivalent to the statement that for every $a\in\HSC$ and every $\Sigma_1$-formula $\phi(x)$ in the language of set theory, if $\forces_\P(\HSC\models\phi(\check{a}))$, then in $\V$, $\HSC\models\phi(a)$.
\end{lem}

The following terminology is from \cite{FuchsMinden:SCforcingTreesGenAbs}.

\begin{defn}
\label{defn:Natural}
If $\P$ is a notion of forcing and $p\in\P$ is a condition, then $\P_{{\le}p}$ is the restriction of $\P$ to conditions $q\le p$. Two forcing notions $\P$ and $\Q$ are \emph{equivalent} if they give rise to the same forcing extensions. A forcing class $\Gamma$ is \emph{natural} if for every $\P\in\Gamma$ and every $p\in\P$, there is a $\Q\in\Gamma$ such that $\P_{{\le}p}$ is forcing equivalent to $\Q$.
\end{defn}

\begin{fact}
\label{fact:BFACharacterizationForNaturalGamma}
If $\Gamma$ is natural, then $\BFA_\Gamma$ is equivalent to the following statement:
\claim{$(*)$}{
for every $\P\in\Gamma$ and every filter $G$ that is $\P$-generic over $\V$, it follows that
\[\kla{\HSC,\in}\prec_{\Sigma_1}\kla{\HSC^{\V[G]},\in}.\]}
\end{fact}

\begin{proof}
Clearly, $(*)$ implies the condition stated in Lemma \ref{lem:BagariaForIndividualPoset}, which is equivalent to the condition stated in Definition \ref{defn:BFA}. For the converse, let $\P$, $G$ be as in $(*)$. Let $\phi(x)$ be a $\Sigma_1$-formula in the language of set theory, $a\in\HSC$, and assume that
$\kla{\HSC^{\V[G]},\in}\models\phi(a)$. Let $p\in G$ force this.
It suffices to show that $\kla{\HSC,\in}\models\phi(a)$. Let $\Q\in\Gamma$ be equivalent to $\P_{{\le}p}$. Then, $\Q$ forces that $\phi(\check{a})$ holds in $\HSC$ (of the forcing extension). Thus, since the bounded forcing axiom for $\Q$ holds, it follows from Lemma \ref{lem:BagariaForIndividualPoset} that $\kla{\HSC,\in}\models\phi(a)$.
\end{proof}

Most commonly encountered forcing classes are natural. In particular, it follows from \cite[Corollary 3.11]{FuchsRinot:WeakSquareStationaryReflection} and \cite[Observation 2.4]{Fuchs:ParametricSubcompleteness} that the class of subcomplete forcing notions is natural.

Here is a version of Corollary \ref{cor:bfMPSC+NoIMwithAnIAlimitOfMeasurablesImpliesWO} with a more restrictive anti-large cardinal assumption, but with \BSCFA{} in place of $\bfMPSC$.

\begin{lem}
\label{lem:BSCFA+NoZeroSharpImpliesWOofOmega1}
Suppose $0^\#$ does not exist, and that $\BSCFA$ holds. Then there is a well-order of $\power(\omega_1)$ of order type $\omega_2$, $\Delta_1$-definable without parameters. This well-order is in $L(\power(\omega_1))$, is $\Delta_1$-definable from a subset $\bI$ of $\omega_1$ there, and it is $\Delta_1^{\kla{H_{\omega_2},\in}}$-definable in $\bI$.
\end{lem}

\begin{proof}
Following the argument in the proof of Corollary \ref{cor:bfMPSC+NoIMwithAnIAlimitOfMeasurablesImpliesWO},
let $\seq{\kappa_i}{i<\omega_1}$ enumerate the first $\omega_1$ many GCH survivors greater than $\omega_1$. Let $Z\in L$ be the $L$-least set of ordinals such that $I:=\{\kappa_i\st i<\omega_1\}\sub Z$ and $Z$ has cardinality $\omega_1$. Let $\pi_Z:Z\To\otp(Z)$ be the Mostowski-collapse of $Z$, and let $\bar{I}=\pi_Z``I$.

For a set of ordinals $x$, let $(x)_i$ be the $i$-th element of $x$ in its monotone enumeration.
Let $A\sub\omega_1$ be given. Then, for some forcing extension $\V'$ of $\V$ by some subcomplete forcing, the statement $\phi(A,\bI)=$``there is a set of ordinals $x\in L[\bI]$ such that for all $i<\omega_1$, $i\in A$ iff $\cf((x)_i)=\omega_1$, and $i\notin A$ iff $\cf((x)_i)=\omega$'' holds in $H_{\omega_2}^{\V'}$. To see this, note that $I\in L[\bI]$, since $I$ is coded by $Z$ and $\bI$, and we can force $\GCH$ up to $\sup(I)$, passing to $\V[G]$, and then use Jensen's extended Namba forcing to code $A$ into the cofinality $\omega$/$\omega_1$-pattern on $I$, thus producing a forcing extension $\V[G][H]$ in which
$\phi(A,\bI)$ holds (as witnessed by $x=I$). In a last step, if necessary, we can force over $\V[G][H]$ with $\Col{\omega_1,\sup(x)}$, where $x$ is the $L[\bI]$-least witness to $\phi(A,\bI)$, reaching $\V'$, so that the truth of $\phi(A,\bI)$ is already visible in $H_{\omega_2}^{\V'}$. It's easy to see that $\phi$ can be written as a $\Sigma_1$ formula. So, by \BSCFA, it follows that $H_{\omega_2}\models\phi(A,\bI)$. This can be done for every $A\sub\omega_1$.
We can thus let $f(A)$ be the $L[\bI]$-least $x$ witnessing that $\phi(A,\bI)$ holds. Then $A<B$ iff $f(A)<f(B)$ is a well-order of $\power(\omega_1)$, and since for every $A$, $f(A)<\omega_2$, the order type of that well-order is $\omega_2$. Since $\bI$ is definable without parameters, this well-order is definable without parameters.

Note that for any $\gamma<\omega_2$, ${}^{\omega_1}\gamma\sub L(\power(\omega_1))$ and that for any $A$, the $L[\bI]$-least $x$ witnessing that $\phi(A,\bI)$ holds already exists in $L[\bI]^{H_{\omega_2}}$. It follows that $L(\power(\omega_1))$ is correct about the cofinality of $(x)_i$, for $i<\omega_1$, and hence that in $L(\power(\omega_1))$, the well-ordering of $\power(\omega_1)$ described above is definable from $\bI$.
\end{proof}

Paralleling the treatment of $\bfMPSC$ in the previous section, let's draw a conclusion which is related to the question which forcing notions preserve \BSCFA. It's easy to see that subcomplete forcing notions that don't add subsets of $\omega_1$ (and thus don't change $H_{\omega_2}$) preserve \BSCFA. One might hope that the requirement of not adding subsets of $\omega_1$ can be dropped, but the following lemma shows that this is not the case, at least in the absence of $0^\#$, and if one insists that the forcing preserve $\omega_2$.

\begin{lem}
\label{lem:NonPreservationOfBSCFA}
Assume $0^\#$ does not exist. Suppose $N$ is a set-forcing extension of $\V$ with $\omega_2=\omega_2^N$, and that $N$ satisfies \BSCFA. Then $\power(\omega_1)^\V=\power(\omega_1)^N$.
\end{lem}

\begin{proof}
Observe that $\omega_1^\V=\omega_1^N$. Let $N=\V[g]$, where $g$ is generic for $\P$ over $\V$. Let $\seq{\kappa_i}{i<\omega_1}$ enumerate the next $\omega_1$ many GCH survivors greater than the cardinality of $\P$. Then every $\kappa_i$ is a GCH survivor in $N$ as well. Let $I:=\{\kappa_i\st i<\omega_1\}\sub Z\in L$, $Z\sub\On$, and let $\bI=\pi_Z``I$, where, as before, $\pi_Z$ is the Mostowski-collapse of $Z$.

Let $A\sub\omega_1$, $A\in N$. I have to show that $A\in\V$. Working in $N$, the argument of the proof of Lemma \ref{lem:BSCFA+NoZeroSharpImpliesWOofOmega1} shows that there is a set of ordinals $x\in L[\bI]$, $x\sub\omega_2$, such that for all $i<\omega_1$, $i\in A$ iff $\cf((x)_i)=\omega_1$, and $i\notin A$ iff $\cf((x)_i=\omega)$. But note that if $\gamma<\omega_2$, then $\cf(\gamma)^\V=\cf(\gamma)^N$, $\omega_1^\V=\omega_1^N$. Hence, $A$ can be defined from $x$ in $\V$.
\end{proof}

It is possible to produce a forcing extension of a model of $\BSCFA+\neg 0^\#$, preserving $\omega_1$, and adding subsets of $\omega_1$, to reach another model of \BSCFA, but collapsing $\omega_2$. To see this, recall the concept of a reflecting cardinal: a regular cardinal $\kappa$ is reflecting if for every formula $\phi(x)$ and every $a\in H_\kappa$, if there is a cardinal $\gamma>\kappa$ such that $\kla{H_\gamma,\in}\models\phi(a)$, then there is a cardinal $\bgamma<\kappa$ such that $a\in H_\bgamma$ and $\kla{H_\bgamma,\in}\models\phi(a)$.
Reflecting cardinals were introduced in \cite{GoldsternShelah:BPFA}, where it was shown that the consistency strength of \BPFA{} is precisely a reflecting cardinal. This was extended in \cite{Fuchs:HierarchiesOfForcingAxioms} to \BSCFA. In detail, it was shown there that $\BSCFA$ implies that $\omega_2$ is reflecting in $L$, and that if $\kappa$ is reflecting, then there is a subcomplete forcing $\P$ that's $\kappa$-c.c., has size $\kappa$, collapses $\kappa$ to become $\omega_2$, and such that $\BSCFA$ holds in $\V^\P$. Thus, if we start in a model of set theory with two reflecting cardinals, $\kappa<\delta$, in which we may assume $0^\#$ does not exist, then we may use $\kappa$ to reach a forcing extension $M$ in which $\BSCFA$ holds, $\delta$ still is reflecting, and $0^\#$ does not exist. Now we can use the reflecting cardinal $\delta$ in $M$ to force $\BSCFA$ again, collapsing $\omega_2$.
So the requirement that $\omega_2$ be preserved in the previous lemma is necessary.

It is obvious that in the previous two lemmas, the assumption that $0^\#$ does not exist can be replaced with the weaker assumption that there is some $x\sub\omega$ such that $x^\#$ does not exist. 

\section{More reflection, or: what is the bounded forcing axiom for countably closed forcing?}
\label{sec:MoreReflection}

The most obvious way to try to obtain the consequences of Theorem \ref{thm:WOofPomega1FrombfMPSCandSetManyGrounds}, with the assumption of $\bfMPSC$ weakened to a form of the bounded forcing axiom, would seem to be to replace $\Sigma_1$-elementarity in Fact \ref{fact:BFACharacterizationForNaturalGamma} with $\Sigma_2$-elementarity. This motivates the following definition. I will analyze the resulting principles, and propose an answer to the question stated in the section title.

\begin{defn}
\label{defn:SigmaNBFAGamma}
Let $\Gamma$ be a natural forcing class and $n\in\omega$. Then the principle $\Sigma_n\text{-}\BFA_\Gamma$ says that whenever $G$ is generic for some $\P\in\Gamma$, it follows that
\[\kla{\HSC,\in}\prec_{\Sigma_n}\kla{\HSC,\in}^{\V[G]}.\]
\end{defn}

However, if the forcing class in question contains the class of ccc forcing notions, then the resulting principle is inconsistent, for $n\ge 2$.

\begin{obs}
\label{obs:InconsistencyOfSigma2-BFAccc}
Let $\Gamma$ be a natural forcing class.
\begin{enumerate}[label=(\arabic*)]
  \item
  \label{item:NewRealsImplyNonCH}
  If there is a $\P\in\Gamma$ that necessarily adds a real, then $\BFA_\Gamma$ implies the failure of $\CH$.
  \item
  \label{item:BFAcccImpliesSH}
  $\BFA_\ccc$ implies Souslin's hypothesis, i.e., that there is no Souslin tree.
  \item
  \label{item:Sigma2-BFAcccInconsistent}
  $\Sigma_2\text{-}\BFA_\ccc$ is inconsistent.
\end{enumerate}
\end{obs}

\begin{proof}
For \ref{item:NewRealsImplyNonCH}, assume $\BFA_\Gamma$, and suppose $\P\in\Gamma$ adds a real. Assume, towards a contradiction, that $\CH$ holds. Then $a=\power(\omega)\in\HSC$. If $G$ is $\P$-generic, then the $\Sigma_1$-statement $\phi(a)$, expressing ``there is an $x\sub\omega$ with $x\notin a$'' holds in $\kla{\HSC,\in}^{\V[G]}$, but not in $\kla{\HSC,\in}^\V$, a contradiction.

\ref{item:BFAcccImpliesSH} is clear, because if $T$ were a Souslin tree, then, viewing $T$ as a notion of forcing in the usual way, $T$ is ccc, and if $b\sub T$ is $T$-generic, then the $\Sigma_1$-statement ``there is a function $f:\omega\To T$ such that for all $\alpha<\beta<\omega_1$, $f(\alpha)<_Tf(\beta)$'' in the parameters $T$ and $\omega_1$ holds in $\kla{\HSC,\in}^{\V[G]}$, but of course not in $\kla{\HSC,\in}^\V$.

\ref{item:Sigma2-BFAcccInconsistent} now follows because a Souslin tree can be added by ccc forcing (for example, Cohen forcing adds a Souslin tree). But if $\V[G]$ has a Souslin tree, then this can be expressed as a $\Sigma_2$ statement in $\kla{\HSC,\in}^{\V[G]}$. But then, $\Sigma_2\text{-}\BFA_\ccc$ would imply that there is a Souslin tree in $\V$, contradicting \ref{item:BFAcccImpliesSH}.
\end{proof}

The situation with countably closed forcing is different, though. It is well-known, and easy to see, that the full version of Martin's axiom for countably closed forcing is provable in \ZFC: if $\P$ is countably closed and $\mathcal{D}$ is a collection of $\omega_1$ many maximal antichains in $\P$, then there is a filter in $\P$ that meets each antichain. Hence, forcing axioms for countably closed forcing have not been considered, with the exception of the ``$+$''-versions, introduced in \cite{FMS:MM1}.

I will argue that the $\Sigma_2\text{-}\BFA_\CC$ is the correct version of the bounded forcing axiom for countably closed forcing. First, note that the class of countably closed forcing notions is natural. It thus follows from Fact \ref{fact:BFACharacterizationForNaturalGamma} that:

\begin{fact}
\label{fact:Sigma1AbsolutenessForCC}
Whenever $\P$ is a countably closed forcing notion and $G$ is $\P$-generic, then
\[\kla{\HSC,\in}\prec_{\Sigma_1}\kla{\HSC,\in}^{\V[G]}.\]
\end{fact}

Thus, the axiom $\Sigma_2\text{-}\BFA_\CC$ says that we have one more level of absoluteness than \ZFC guarantees. This is what $\BFA_\Gamma$ says for the classes of ccc or proper forcing as well. In this sense, $\Sigma_2\text{-}\BFA_\CC$ seems to be a good candidate for the ``correct'' version of the bounded forcing axiom for countably closed forcing. Another requirement is of course that it should be consistent, from adequate large cardinal assumptions. Recall that it was shown in Goldstern-Shelah \cite{GoldsternShelah:BPFA} that the consistency strength of the bounded proper forcing axiom is a reflecting cardinal.

\begin{defn}[{\cite[Def.~2.2]{GoldsternShelah:BPFA}}]
\label{def:Reflecting}
A regular cardinal $\kappa$ is \emph{reflecting} if for every $a\in H_\kappa$, and every formula $\phi(x)$, the following holds: if there is a regular cardinal $\theta\ge\kappa$ such that $H_\theta\models\phi(a)$, then there is a cardinal $\bar{\theta}<\kappa$ such that $H_{\btheta}\models\phi(a)$.
\end{defn}

I showed in \cite{Fuchs:HierarchiesOfForcingAxioms} that the consistency strength of \BSCFA is also a reflecting cardinal. The following theorem thus supports very strongly the claim that the axiom $\Sigma_2\text{-}\BFA_\CC$ is the correct version of the bounded forcing axiom for countably closed forcing.

\begin{thm}
\label{thm:ConsistencyStrengthOfSigma2BFACC}
The consistency strength of $\Sigma_2\text{-}\BFA_\CC$ is a reflecting cardinal. More precisely:
\begin{enumerate}[label=(\arabic*)]
\item
\label{item:ForcingDirection}
If $\kappa$ is a reflecting cardinal and $G$ is generic for $\Col{\omega_1,{{<}\kappa}}$, then the principle $\Sigma_2\text{-}\BFA_\CC$ holds in $\V[G]$.
\item
\label{item:InnerModelDirection}
The axiom $\Sigma_2\text{-}\BFA_\CC$ implies that $\omega_2^\V$ is a reflecting cardinal in $L$.
\end{enumerate}
\end{thm}

\begin{proof}
For \ref{item:ForcingDirection}, let $\kappa$ and $G$ be as described. In $\V[G]$, let $\Q$ be a countably closed forcing notion, and let $H$ be $\Q$-generic over $\V[G]$. Let $a\in\HSC^{\V[G]}$, let $\phi(x)$ be a $\Sigma_2$-formula, and suppose that $\kla{\HSC,\in}^{\V[G][H]}\models\phi(a)$. Since $\kappa$ is inaccessible, $\Col{\omega_1,{<}\kappa}$ is $\kappa$-cc, and it follows that there is some $\alpha<\kappa$ such that if we let $G_\alpha=G\cap\Col{\omega_1,{<}\alpha}$, then $a\in\HSC^{\V[G_\alpha]}$. Let $G_{[\alpha,\kappa)}=G\cap\Col{\omega_1,[\alpha,\kappa)}$.

It is well-known that $H$ can be absorbed into a collapse, meaning that in $\V[G]$, there is a regular cardinal $\tau$ such that $\Col{\omega_1,[\kappa,\tau)}$ is forcing equivalent to $\Q\times\Col{\omega_1,[\kappa,\tau)}$. Thus, if we let $I$ be $\Col{\omega_1,[\kappa,\tau)}$-generic over $\V[G][H]$, then there is an $I^*$ which is generic over $\V[G]$ for $\Col{\omega_1,[\kappa,\tau)}$ such that $\V[G][H][I]=\V[G][I^*]=\V[G_\alpha][G_{[\alpha,\kappa)}][I^*]$. Let $\phi(x)=\exists y \psi(x,y)$, where $\psi(x,y)$ is a $\Pi_1$-formula, and let $b\in\HSC^{\V[G][H]}$ be such that $\HSC^{\V[G][H]}\models\psi(a,b)$. Then by Fact \ref{fact:Sigma1AbsolutenessForCC}, it follows that $\HSC^{\V[G][H][I]}\models\psi(a,b)$ as well, in particular, $\HSC^{\V[G][H][I]}\models\phi(a)$. Since $\V[G][H][I]=\V[G_\alpha][G_{[\alpha,\kappa)}][I^*]$, we have that
\[\HSC^{\V[G_\alpha][G_{[\alpha,\kappa)}][I^*]}\models\phi(a).\]
Now, working in $\V[G_\alpha]$, let $\theta>\kappa$ be a regular cardinal such that in $H_\theta^{\V[G_\alpha]}$, there is a condition $p$ in $G_{[\alpha,\kappa)}\times I^*$ that forces with respect to $\Col{\omega_1,[\alpha,\kappa)}\times\Col{\omega_1,[\kappa,\tau)}\isomorphic\Col{\omega_1,[\alpha,\tau)}$ that in the extension, it is true that $\kla{\HSC,\in}\models\phi(\check{a})$ holds. Actually, since $\Col{\omega_1,[\alpha,\tau)}$ is weakly homogeneous, it follows that the empty condition already forces this. By reflection, there is now a regular cardinal $\theta$ such that in $H_\theta^{\V[G_\alpha]}$, the it is the case that there is a regular cardinal $\tau'$ such that $\Col{\omega_1,[\alpha,\tau')}$ forces that $\phi(\check{a})$ holds in the structure $\kla{\HSC,\in}$, as computed in the extension. Since $\kappa$ is still reflecting in $\V[G_\alpha]$, it follows that there is a regular cardinal $\theta'<\kappa$ such that the same statement is true in $H_{\theta'}^{\V[G_\alpha]}$. Letting $\tau'$ witness this, it then follows that $\phi(a)$ holds in $\HSC^{\V[G_\tau')}$, and as before, this persists to $\HSC^{\V[G]}$, by Fact \ref{fact:Sigma1AbsolutenessForCC}. Thus, $\kla{\HSC,\in}^{\V[G]}\models\phi(a)$, as wished.

For \ref{item:InnerModelDirection}, let $\kappa=\omega_2^\V$. Clearly then, $\kappa$ is a regular cardinal in $L$. To show that it is reflecting in $L$, let $\theta>\kappa$ be a regular cardinal in $L$, $a\in H_\kappa^L=L_\kappa$, and $\phi(a)$ a formula that holds in $H_\theta^L=L_\theta$. Let $G$ be generic for $\Col{\omega_1,\theta}$. Then $\omega_2^{\V[G]}>\theta$, and in $\kla{\HSC,\in}^{\V[G]}$, the statement $\psi(a)$ expressing the following holds: ``there is a $\btheta$ such that $\btheta$ is regular in $L$ and such that $\phi(a)$ holds in $\kla{L_\btheta,\in}$.'' Saying that $\btheta$ is regular in $L$ is equivalent to saying that $\btheta$ is regular in $L_{\omega_2^{\V[G]}}=L^{\HSC^{\V[G]}}$, and this can be expressed in $\HSC^{\V[G]}$ by saying that for every $f$, if $f\in L$, then if $f$ is a function from some $\gamma<\btheta$ to $\btheta$, the range of $f$ is bounded in $\btheta$. Saying that $f\in L$ is a $\Sigma_1$ statement, so the conditional is $\Pi_1$, and it is thus easily seen that $\psi(a)$ can be chosen to be a $\Sigma_2$-formula. Thus, by $\Sigma_2\text{-}\BFA_\CC$, it follows that $\kla{\HSC,\in}^\V\models\psi(a)$, and if we let $\btheta$ witness this, then $\btheta<\kappa$ is regular in $L$ and $\kla{L_\btheta,\in}\models\phi(a)$, as wished.
\end{proof}

The following observation completes the picture, illustrating that $\Sigma_2\text{-}\BFA_\CC$ plays the role in the context of countably closed forcing that $\BFA_\Gamma$ (equivalently, $\Sigma_1\text{-}\BFA_\Gamma$) played in the case where $\Gamma$ is the class of proper or subcomplete forcing notions, and parallels Observation \ref{obs:InconsistencyOfSigma2-BFAccc}, with the role the Souslin trees used to play taken over by Kurepa trees.

\begin{obs}
\label{obs:InconsistencyOfSigma3-BFACC}
The following facts hold about the principles $\Sigma_n\text{-}\BFA_\CC$.
\begin{enumerate}[label=(\arabic*)]
  \item
  \label{item:GettingDiamond}
  The principle $\Sigma_2\text{-}\BFA_\CC$ implies $\CH$, and even $\diamondsuit$.
  \item
  \label{item:NonKH}
  Furthermore, $\Sigma_2\text{-}\BFA_\CC$ implies the failure of Kurepa's hypothesis, that is, it implies that there are no Kurepa trees.
  \item
  \label{item:Sigma3BFACCinconsistent}
  The principle $\Sigma_3\text{-}\BFA_\CC$ is inconsistent.
\end{enumerate}
\end{obs}

\begin{proof}
\ref{item:GettingDiamond} follows because the principle $\diamondsuit$ can be forced by countably closed forcing, for example by $\Add(\omega_1,1)$, and it is easy to see that $\diamondsuit$ can be expressed by a $\Sigma_2$ sentence in $\HSC$, using $\omega_1$ as a parameter: there is a sequence $\seq{D_\alpha}{\alpha<\omega_1}$ such that for every set $A\sub\omega_1$ and every club set $C\sub\omega_1$, there is an $\alpha\in C$ such that $A\cap\alpha=C_\alpha$.

For \ref{item:NonKH}, suppose $T\in\HSC$ were a Kurepa tree. Let $\kappa$ be the cardinality of the set of cofinal branches through $T$. Then after forcing with $\Col{\omega_1,\kappa}$, say to reach $\V[G]$, $T$ is no longer a Kurepa tree, because $\Col{\omega_1,\kappa}$, and more generally, no countably closed forcing, can add a cofinal branch to $T$. But the statement that $T$ is a Kurepa tree can be expressed over $\HSC$ by a $\Pi_2$ formula $\phi(T)$, essentially saying that $T$ is an $\omega_1$ tree such that for every set $x$, there is a cofinal branch through $T$ that's not in $x$ (since every set in $\HSC$ has size at most $\omega_1$). Thus, in this scenario, it is not true that $\kla{\HSC,\in}\prec_{\Sigma_2}\kla{\HSC,\in}^{\V[G]}$.

For \ref{item:Sigma3BFACCinconsistent}, assuming $\Sigma_3\text{-}\BFA_\CC$, we know by \ref{item:NonKH} that there is no Kurepa tree. But it is well-known that a Kurepa tree may be added by a countably closed poset. Let $\V[G]$ be obtained by forcing with such a poset. Then, since we have just seen in the proof of \ref{item:NonKH} that ``$T$ is Kurepa'' is a $\Pi_2$ statement about $T$ in $\HSC$, the statement ``there is a Kurepa tree'' is expressed by a true $\Sigma_3$ sentence over $\HSC^{\V[G]}$. So by $\Sigma_3\text{-}\BFA_\CC$, it follows that there is a Kurepa tree in $\V$ after all, a contradiction.
\end{proof}

A similar analysis can be carried out for the class of ${<}\kappa$-closed forcing notions, for some regular cardinal $\kappa>\omega_1$. The adequate ``bounded forcing axiom'' for this class would then say that whenever $G$ is generic for some ${<}\kappa$-closed forcing notion, then \[\kla{H_{\kappa^+},\in}\prec_{\Sigma_2}\kla{H_{\kappa^+},\in}^{\V[G]}.\]
If we slightly abuse notation and denote the resulting principle $\Sigma_2\text{-}\BFA_{\closed{\kappa}}$, then the version of Observation \ref{obs:InconsistencyOfSigma3-BFACC} reads:

\begin{obs}
\label{obs:InconsistencyOfSigma3-BFAkappaC}
Let $\kappa\ge\omega_1$ be a regular cardinal.
\begin{enumerate}[label=(\arabic*)]
  \item
  \label{item:GettingDiamondKappa}
  The principle $\Sigma_2\text{-}\BFA_{\closed{\kappa}}$ implies $2^{{<}\kappa}=\kappa$, and even $\diamondsuit_\kappa$.
  \item
  \label{item:NonSlimKHKappa}
  Furthermore, $\Sigma_2\text{-}\BFA_{\closed{\kappa}}$ implies that there are no slim $\kappa$-Kurepa trees.
  \item
  \label{item:Sigma3BFAkappaCinconsistent}
  The principle $\Sigma_3\text{-}\BFA_{\closed{\kappa}}$ is inconsistent.
\end{enumerate}
\end{obs}

For details concerning slim $\kappa$-Kurepa trees in this context, see \cite[Lemma 3.2, Theorem 3.3]{Fuchs:MPclosed}. The consistency strength analysis carries over as well, as follows. The version of Fact \ref{fact:Sigma1AbsolutenessForCC} for ${<}\kappa$-closed forcing does not follow from Fact \ref{fact:BFACharacterizationForNaturalGamma}, but instead, one can appeal to \cite[p.~298, (I6)]{Kunen:SetTheory} and the argument of \cite[Observation 4.19]{FuchsMinden:SCforcingTreesGenAbs}, under the assumption that $2^{{<}\kappa}=\kappa$. The following theorem can then be proven, using the argument of the proof of Theorem \ref{thm:ConsistencyStrengthOfSigma2BFACC}, mutatis mutandis.

\begin{thm}
\label{thm:ConsistencyStrengthOfSigma2BFAkappaC}
Let $\kappa$ be a regular cardinal. Then the consistency strength of $\Sigma_2\text{-}\BFA_{\closed{\kappa}}$ is a reflecting cardinal, in the following sense:
\begin{enumerate}[label=(\arabic*)]
\item
\label{item:ForcingDirectionKappa}
If $\theta>\kappa$ is a reflecting cardinal and $G$ is generic for $\Col{\kappa,{{<}\theta}}$, then the principle $\Sigma_2\text{-}\BFA_{\closed{\kappa}}$ holds in $\V[G]$.
\item
\label{item:InnerModelDirectionKappa}
The axiom $\Sigma_2\text{-}\BFA_{\closed{\kappa}}$ implies that $(\kappa^+)^\V$ is a reflecting cardinal in $L$.
\end{enumerate}
\end{thm}

Some open questions from \cite{Fuchs:MPclosed} translate to open questions about these bounded forcing axioms. For example, is the principle $\Sigma_3\text{-}\BFA_{\directed{\kappa}}$ consistent, assuming the consistency of large cardinals? The argument that works for \closed{\kappa} forcing does not go through for ${<}\kappa$-directed closed forcing, because it is not generally true that one can add a slim $\kappa$-Kurepa tree by ${<}\kappa$-directed closed forcing.

Returning to subcomplete forcing, it is an interesting question whether $\Sigma_2\text{-}\BSCFA$ (that is $\Sigma_2\text{-}\BFA_\Gamma$, where $\Gamma$ is the class of subcomplete forcing notions) is consistent. Here is a consistency strength lower bound.

\begin{cor}
\label{cor:Sigma2BSCFAisStrong}
If $\Sigma_2\text{-}\BSCFA$ holds, then every real has a sharp.
\end{cor}

\begin{proof}
Suppose there was some $r\sub\omega$ such that $r^\#$ does not exist. By the remark at the end of Section \ref{sec:BSCFA}, it follows by Lemma \ref{lem:BSCFA+NoZeroSharpImpliesWOofOmega1} that there is a set $\bar{I}\sub\omega_1$ such that in $\HSC$, a well-order of $\power(\omega_1)$ can be defined from $\bar{I}$ in a $\Delta_1$ way. Now let $A\sub\omega_1$ be generic for $\Add(\omega_1,1)$. By \BSCFA, we know that \CH holds (see Observation \ref{obs:InconsistencyOfSigma3-BFACC}.\ref{item:GettingDiamond}), so that $\Add(\omega_1,1)$ has size $\omega_1$ and hence preserves $\omega_2$. Now in $\HSC^\V$, there is a $\Sigma_1$-definable function $F$ from $\omega_2$ onto $\power(\omega_1)$, and the statement that for every subset $x$ of $\omega_1$, there is an $\alpha$ such that $x=F(\alpha)$ is a $\Pi_2$ statement using the parameter $\bar{I}$. Thus, by $\Sigma_2\text{-}\SCFA$, that same statement holds in $\HSC^{\V[A]}$. So let $\alpha<\omega_2^{\V[A]}=\omega_2^\V$ be such that $A=F^{\HSC^{\V[A]}}(\alpha)$. It then follows that $F^{\HSC^{\V[A]}}(\alpha)=F^{\HSC^\V}(\alpha)$, so that $A\in\HSC^\V$, a contradiction.
\end{proof} 

\section{Enhanced bounded forcing axioms}
\label{sec:EBSCFA}

It would be desirable to prove versions of Theorem \ref{thm:WOofPomega1FrombfMPSCandSetManyGrounds} and Corollary \ref{cor:bfMPSC+NoIMwithAnIAlimitOfMeasurablesImpliesWO} for a version of $\BSCFA$ instead of $\bfMPSC$.
$\bfMPSC$ was needed because the complexity of the forcing invariant inner model used exceeded what can be expressed in a $\Sigma_1$ way inside $H_{\omega_2}$. The previous section showed that this problem cannot be resolved simply by working with $\Sigma_n-\BSCFA$: for $n=2$, it is unclear whether this principle is consistent, and for $n\ge 3$, it is inconsistent. Moreover, the forcing invariant inner model used in the earlier arguments might not even be locally definable in $\HSC$.
So the idea is to formulate a slightly strengthened form of $\BSCFA$, where $H_{\omega_2}$ is equipped with the requisite knowledge about $\V$.

\begin{defn}
\label{defn:BFA(I)}
Let $I$ be a class term, using parameters from $H_{\omega_2}$. Let $\Gamma$ be a class of forcing notions. Then $\BFA_\Gamma(I)$ says that whenever $\P$ is a forcing notion in $\Gamma$, $a\in H_{\omega_2}$ and $\phi(x)$ is a $\Sigma_1$-formula in the language of set theory with an extra predicate symbol $\dot{I}$ such that $\P$ forces that in the extension, $\kla{\HSC,\in,I\cap\HSC}\models\phi(a)$ holds (equivalently, whenever $G$ is $\P$-generic over $\V$, then $\kla{\HSC^{\V[G]},\in,I^{\V[G]}\cap \HSC^{\V[G]}}\models\phi(a)$), then $\kla{\HSC^\V,\in,I^\V\cap \HSC^\V}\models\phi(a)$.
\end{defn}

\begin{fact}
\label{fact:BFA(I)CharacterizationForNaturalGamma}
If $\Gamma$ is natural, then $\BFA_\Gamma(I)$ is equivalent to the following statement:
\claim{$(*)$}{
for every $\P\in\Gamma$ and every filter $G$ that is $\P$-generic over $\V$, it follows that
\[\kla{\HSC,\in,I\cap\HSC}\prec_{\Sigma_1}
\kla{\HSC^{\V[G]},\in,I^{\V[G]}\cap\HSC^{\V[G]}}.\]}
\end{fact}

\begin{proof}
The proof of Fact \ref{fact:BFACharacterizationForNaturalGamma} goes through.
\end{proof}




The idea of enhancing the structure $\HSC$ with a predicate in order to strengthen the bounded forcing axiom is not new. Using this terminology, for example, $\BMM^{++}$ can be expressed equivalently as $\BMM(\NS_{\omega_1})$, see \cite[Lemma 10.94]{Woodin:ADforcingAxiomsNonstationaryIdeal}.

Not any class term $I$ can be used to enhance bounded forcing axioms, as we shall see. Let's explore some restrictions, and the relationship to maximality principles. After all, what we are looking for is an enhanced bounded forcing axiom that will still have the desired effects on the existence of definable well-orders of $\power(\omega_1)$, while being weaker than the full maximality principle.

\begin{remark}
\label{remark:Immunity}
Let $I$ be a class term, and let $\Gamma$ be a natural forcing class.
\begin{enumerate}[label=(\arabic*)]
  \item
  \label{item:BSCFA(I)impliesIcapHSCimmune}
  If $(*)$ holds, then $I\cap\HSC$ is \emph{immune to $\Gamma$}, meaning that that for any subcomplete forcing $\P$, if $G$ is $\P$-generic over $\V$, then $I\cap H_{\omega_2}=I^{\V[G]}\cap H_{\omega_2}^\V$.
  \item
  \label{item:bfMPSCandIcapHSCimmuneimpliesBSCFA(I)}
  If $\MP_\Gamma(\HSC)$ holds, the definition of $I$ only uses parameters from $\HSC$ (if any) and $\Gamma$-necessarily, $I\cap\HSC$ is immune to $\Gamma$ (meaning that whenever $G$ is generic for a poset in $\Gamma$, then in $\V[G]$, $I^{\V[G]}\cap\HSC^{\V[G]}$ is immune to $\Gamma^{\V[G]}$), then $\BFA_\Gamma(I)$ holds.
\end{enumerate}
\end{remark}

\begin{proof}
For \ref{item:BSCFA(I)impliesIcapHSCimmune}, let $G$ be $\P$-generic, where $\P\in\Gamma$. By $(*)$, $\kla{\HSC,\in,I\cap\HSC}\prec_{\Sigma_1}
\kla{\HSC^{\V[G]},\in,I^{\V[G]}\cap\HSC^{\V[G]}}$. This clearly implies that $I\cap\HSC=I^{\V[G]}\cap\HSC$.

To see \ref{item:bfMPSCandIcapHSCimmuneimpliesBSCFA(I)}, let $G$ be generic for some $\P\in\Gamma$, let $a\in\HSC$, and suppose that $\kla{\HSC^{\V[G]},\in,I^{\V[G]}\cap\HSC^{\V[G]}}\models\phi(a)$, where $\phi(x)$ is a $\Sigma_1$-formula in the language of set theory with an extra predicate symbol. We have to show that already in $\V$, $\kla{\HSC,\in,I}\models\phi(a)$ holds.

By assumption, $I^{\V[G]}$ is immune to $\Gamma^{\V[G]}$ in $\V[G]$. This implies that whenever $H$ is generic over $\V[G]$ for some forcing $\Q\in\Gamma^{\V[G]}$, then \[\kla{\HSC^{\V[G][H]},\in,I^{\V[G][H]}\cap\HSC^{\V[G][H]}}\models\phi(a),\] because $\phi(x)$ is $\Sigma_1$ and $\HSC^{\V[G]}$ is a transitive subset of $\HSC^{\V[G][H]}$. Hence, the statement ``$\kla{\HSC,\in,I}\models\phi(a)$'' is $\Gamma$-necessary in $\V[G]$. Since the definition of $I$ only requires parameters from $\HSC$, $\MP_\Gamma(\HSC)$ applies, and it follows that already in $\V$ it is the case that $\kla{\HSC,\in,I}\models\phi(a)$, as claimed.
\end{proof}

Note that $\NS_{\omega_1}$ is immune with respect to stationary set preserving forcing notions, and so, if $\Gamma$ is a class of $\Gamma$-necessarily stationary set preserving forcing notions, then $\MP_\Gamma(H_{\omega_2})$ implies $\BFA_\Gamma(\NS_{\omega_1})$, or what one might call $\BFA^{++}_\Gamma$.

If the class $I$ is definable in $\HSC$, then clearly, $\BFA_\Gamma(I)$ can be viewed as carefully strengthening the elementarity stated in Lemma \ref{fact:BFACharacterizationForNaturalGamma}.

Returning to the enhanced bounded forcing axioms of the form $\BFA_\Gamma(I)$, the following example arises from considering the proof of Observation \ref{obs:InconsistencyOfSigma3-BFACC}.

\begin{example}
Let $\calT$ be the set of all $\omega_1$-Kurepa trees whose nodes are countable ordinals. Then $\BFA_\CC(\calT)$ is inconsistent.
\end{example}
%
%

Note that the class of countably closed forcing notions is contained in any of the forcing classes of interest here, such as the subcomplete, proper, or stationary set preserving forcing notions. Thus, the enhanced bounded forcing axiom for this class, $\BFA_{\CC}(I)$, is the weakest one. The example hence shows that one has to be careful in choosing the class term $I$ by which one wants to enhance the bounded forcing axiom.

The argument in the proof of Observation \ref{obs:InconsistencyOfSigma3-BFACC} also shows that $\MP_\CC(\HSC)$ implies that $\calT=\leer$, and that $\calT$ is immune to $\sigma$-closed forcing, but $\calT$ is never $\sigma$-closed-necessarily immune to $\sigma$-closed forcing. Hence, this extra assumption in Remark \ref{remark:Immunity}.\ref{item:bfMPSCandIcapHSCimmuneimpliesBSCFA(I)} can't be dropped, and it does not follow automatically.

Here is a version of Lemma \ref{lem:WOfromTechnicalAssumption} for the context of enhanced bounded subcomplete forcing axioms.

\begin{lem}
\label{lem:WOfromTechnicalAssumptionWithBSCFA(I)}
Assume that $\psi(x,y,z)$ is a formula in the language of set theory such that for some set $P\sub\omega_1$ and some $\alpha\in\On$,
\[I_{\alpha,P}=\{\beta\in\On\st\psi(\beta,\alpha,P)\}\]
has size at least $\omega_1$ and consists of GCH survivors, and that $I_{\alpha,P}$ is absolute to subcomplete forcing extensions. Let
\[I_P=\{\kla{\gamma,\delta}\st \psi(\gamma,\delta,P)\}\]
Then $\BSCFA(I_P)$ implies the existence of a well-order of $\power(\omega_1)$, definable from $P$. This well-order is $\Delta_1^{\kla{H_{\omega_2},\in,I_P}}$
\end{lem}

\begin{proof}
Given a set $A\sub\omega_1$, and fixing $P$ and $\alpha$ as above, let $\phi(\alpha,A,P)$ be the statement saying ``for all $i<\omega_1$, $i\in A$ iff $\cf((I_{\alpha,P})_i)=\omega_1$ and $i\notin A$ iff $\cf((I_{\alpha,P})_i)=\omega$''. The statement $\phi(A,P)$, saying ``there is an $\alpha'$ such that $\phi(\alpha',A,P)$ holds'' is then true in $\kla{H_{\omega_2},\in,I\cap\HSC}^{\V[G]}$, where $G$ is generic for the subcomplete forcing to code $A$ into the cofinality $\omega$/$\omega_1$-pattern and to collapse $\alpha$ to $\omega_1$. Since $\phi$ is $\Sigma_1$ (using $I$ as a predicate), it follows from $\BSCFA(I_P)$ that $\phi(A,P)$ is true in $\V$. So we can let $f(A)$ be the least $\alpha'$ such that $\phi(\alpha',A,P)$ holds and say that $A<B$ iff $f(A)<f(B)$.
\end{proof}

Let $\bbC=\{\alpha\st\alpha\ \text{is a GCH survivor}\}$. The strengthening of $\BSCFA$ I'll be mostly interested in is $\BSCFA(\bbC^\Mantle)$. That is, the class used to enhance \BSCFA is the relativization of the class of all GCH survivors to the mantle $\Mantle$. Obviously, since $\Mantle$ is forcing-absolute, so is $\bbC^\Mantle$. In particular, subcomplete-necessarily, $\bbC^\Mantle\cap\HSC$ is immune to subcomplete forcing, and since the definition of $\bbC^\Mantle\cap\HSC$ needs no parameters, it follows from Remark \ref{remark:Immunity}.\ref{item:bfMPSCandIcapHSCimmuneimpliesBSCFA(I)}
that $\bfMPSC$ implies $\BSCFA(\bbC^\Mantle)$. The point of $\BSCFA(\bbC^\Mantle)$ is that it has the same consequences as $\bfMPSC$, in terms of the existence of definable well-orders of $\power(\omega_1)$.

\begin{lem}
\label{lem:WOfromBSCFA(C^M)andSetManyGrounds}
If there are only set many grounds and $\BSCFA(\bbC^\Mantle)$ holds,
then there is a well-order of $\power(\omega_1)$ that's $\Delta_1^{\kla{H_{\omega_2},\in,\bbC^\Mantle\cap H_{\omega_2}}}$-definable, and $2^{\omega_1}=\omega_2$. This holds also with the weakened assumption that there is an $\alpha$ such that $\bbC\ohne\alpha=\bbC^\Mantle\ohne\alpha$.
\end{lem}

\begin{proof}
Let $\psi(x,y,z)$ express that $x$ and $y$ are ordinals, and that $x>y$ is a GCH survivor in \Mantle{} (the variable $z$ is not used). If there are only set-many grounds, then if $\alpha$ is at least as large as the size of the forcing leading from \Mantle{} to $\V$, then, in the notation of Lemma \ref{lem:WOfromTechnicalAssumptionWithBSCFA(I)}, $I_{\alpha,\leer}$ satisfies the assumptions made in that lemma. It follows that $\BSCFA(I_\leer)$ has the consequences claimed. But clearly, $I_\leer$ is definable from $\bbC^\Mantle$ in a very simple way, so that the conclusion follows from $\BSCFA(\bbC^\Mantle)$.
\end{proof}

The following corollary is derivable as before.

\begin{cor}
Assume that there are only set-many grounds, $N$ is a set-forcing extension of $\V$, $\omega_2^\V=\omega_2^N$ and $N\models\BSCFA(\bbC^\Mantle)$. Then $\power(\omega_1)^\V=\power(\omega_1)^N$.
\end{cor}

Not surprisingly, there is a version of Corollary \ref{cor:bfMPSC+NoIMwithAnIAlimitOfMeasurablesImpliesWO} for an appropriately enhanced bounded subcomplete forcing axiom.

\begin{cor}
Suppose there is no inner model with an inaccessible limit of measurable cardinals, and assume $\BSCFA(K,\Card^K)$ holds. Then there is a definable well-order of $\power(\omega_1)$.
\end{cor}

\begin{proof}
It follows from the assumptions that the set of measurable cardinals of $K$ is bounded in $\delta=\omega_2^\V$. For otherwise, for every $\gamma<\delta$, the statement ``there is a measurable cardinal greater than $\gamma$ in $\dot{K}$'' is true in $\kla{H_{\omega_2},\in,K\cap H_{\omega_2},\Card^K}^{\V^{\Col{\omega_1,\kappa}}}$, where $\kappa>\gamma$ is measurable in $K$. This can be expressed in a $\Sigma_1$ way in this structure, and hence it is already true in $\kla{H_{\omega_2},\in,K\cap H_{\omega_2},\Card^K}$. Thus, $\delta$ is an inaccessible limit of measurable cardinals in $K$. This contradicts our assumption. Thus, using Mitchell's covering lemma, there a subset $D$ of $\omega_1$ that codes a maximal set of indiscernibles for $K$ such that $K[D]$ satisfies Jensen covering. Now, letting $X$ be the next $\omega_1$ many GCH survivors greater than $\omega_1$, it follows as before there is a set $\bX\sub\omega_1$, such that $X\in K[D][\bX]$. Letting $P=C\oplus\bX$, we can now let $\psi(x,y,z)$ be the statement that $x$ and $y$ are ordinals, and that $x$ belongs to the $y$-th element of $\dot{K}[z]$. Then there is an $\alpha$ such that $I_{\alpha,P}$ has size at least $\omega_1$, and the corollary follows from Lemma \ref{lem:WOfromTechnicalAssumptionWithBSCFA(I)}.
\end{proof}

It turns out that $\BSCFA(\bbC^\Mantle)$ is no stronger than $\BSCFA$ (in consistency-strength). In the proof of this fact, I'll use some standard notation: I write $C'$ for the set of all limit points below the supremum of a set $C$ of ordinals, and $S^\lambda_\kappa$ stands for the set of ordinals less than $\lambda$ of cofinality $\kappa$.

\begin{lem}
\label{lem:ConsistencyStrengthOfEnhancedBSCFA}
The consistency strength of $\BSCFA(\bbC^\Mantle)$ is a reflecting cardinal.
\end{lem}

\begin{proof}
Clearly, a reflecting cardinal is a lower bound, because $\BSCFA(\bbC^\Mantle)$ implies $\BSCFA$, and the consistency strength of the latter is a reflecting cardinal, see \cite[Theorem 3.6]{Fuchs:HierarchiesOfForcingAxioms}.

To see that a reflecting cardinal is an upper bound, recall that reflecting cardinals go down to $L$, so we may assume $\V=L$ and that there is a reflecting cardinal. In \cite{Fuchs:HierarchiesOfForcingAxioms}, it was shown that there is then a $\kappa$-c.c.~subcomplete forcing which forces $\BSCFA + \kappa=\omega_2$. Let's call the resulting model $L[g]$. Note that $\Mantle^{L[g]}=L$, and hence, in $L[g]$, $\bbC^{\Mantle}=\bbC^L$ is just the class of ordinals that are regular cardinals in $L$. This remains true in any further forcing extension of $L[g]$. Hence, it suffices to show that $\BSCFA(\bbC^L)$ holds in $L[g]$.

In fact, I will show that in general, $\BSCFA + \neg 0^\#$ implies $\BSCFA(\bbC^L)$.
To see this, assume $\BSCFA + \neg 0^\#$, and let $N=\V[g]$, where $g$ is generic for a subcomplete forcing $\P$. Let $a\in H_{\omega_2}$, and let $\phi(x)$ be a $\Sigma_1$-formula such that $\kla{H_{\omega_2}^N,\in,\bbC^L\cap\omega_2^N}\models\phi(a)$.

Note that if $\rho$ is a cardinal in $L$, then $\bbC^{L_\rho}=\bbC^L\cap\rho$. I'll use a trick I employed in \cite{Fuchs:HierarchiesOfForcingAxioms}, and which goes back to \Todorcevic \cite{Todorcevic:TreesAndLinearlyOrderedSets}, to express that an ordinal is regular in $L$ in a ``$\Sigma_1$ way'', using the canonical global $\square$-sequence $\vec{C}$ of $L$ from Jensen \cite{FS}. That is, \[\vec{C}=\seq{C_\alpha}{\alpha\ \text{is a singular ordinal in}\ L}\]
and for every $L$-singular $\alpha$, $C_\alpha\sub\alpha$ is club, $\otp(C_\alpha)<\alpha$, and if $\beta\in C'_\alpha$, then $\beta$ is singular in $L$ and $C_\beta=C_\alpha\cap\beta$. $\vec{C}$ is $\Sigma_1$-definable in $L$.

Note that every ordinal $\alpha\in(S^{\omega_3}_\omega)^N$ greater than $\omega_2^N$ is singular in $L$, since a subset of $\alpha$ of order type $\omega$ can be covered by a subset of $\alpha$ in $L$ that has order type less than $\omega_2^N$, using our assumption that $0^\#$ does not exist, by Jensen' covering lemma. So, by Fodor's theorem, there are a stationary $A_0\sub(S^{\omega_3}_\omega)^N$, and an ordinal $\alpha_0$ such that \[\forall\alpha\in A_0\quad\otp(C_\alpha)=\alpha_0.\]
Let $C$ be generic for the forcing $\P_{A_0}$ to shoot a club of order type $\omega_1$ through $A_0$, which is subcomplete, by \cite[Lemma 6.3]{Jensen2014:SubcompleteAndLForcingSingapore}. Let $D$ be generic over $N[C]$ for $\Col{\omega_1,\omega_3^N}$. Then, in $H_{\omega_2}^{N[C][D]}$, the following statement holds: ``there are a club $E$ of order type $\omega_1$, with supremum $\delta$, and a $\beta_0$, and a $Y$, and a $c$, such that for all $\alpha\in E$, $\otp(C_\alpha)=\beta_0$, such that $c=\bbC^{L_\delta}\cap Y$, $Y$ is a transitive set, $a\in Y$, and $\phi^{\kla{Y,\in,c}}(a)$'' holds. This is witnessed by $E=C$, $\delta=\omega_3^N$, $\beta_0=\alpha_0$, $Y=H_{\omega_2}^N$. This is a $\Sigma_1$-statement in the parameter $\omega_1^\V$. The point is that the map that sends $\gamma$ to $L_\gamma$ is $\Sigma_1$, and hence, the statement ``$c=\bbC^{L_\delta}$'' is expressible in a $\Sigma_1$ way as well.

Since $N[C][D]$ is a generic extension of $\V$ by a subcomplete forcing notion, the same statement is true in $H_{\omega_2}^\V$, by \BSCFA. Let $E,\delta,\beta_0,Y,c$ be as in the statement. It follows that $\delta$ is a regular cardinal in $L$, because if it weren't, then $C_\delta$ would be defined. $\delta$ has cofinality $\omega_1$, so it would follow that $C'_\delta\cap E$ is club in $\delta$, and for each $\alpha\in C'_\delta\cap E$, we would have that $\otp(C_\delta\cap\alpha)=\otp(C_\alpha)=\beta_0$. But there can be at most one such $\alpha$. This is a contradiction. Hence, $c=\bbC^{L_\delta}=\bbC^L\cap\delta$, and it follows that $\kla{Y,\in,\bbC^L\cap Y}\models\phi(a)$, and hence that $\kla{H_{\omega_2},\in,\bbC^L\cap\omega_2}\models\phi(a)$, because $Y$ is transitive and $\phi$ is $\Sigma_1$. Thus, $\kla{H_{\omega_2},\in,\bbC^L\cap\omega_2}\prec_{\Sigma_1}\kla{H_{\omega_2}^N,\in,\bbC^L\cap\omega_2^N}$, as desired.
\end{proof}

Let's now try to find an appropriate strengthening of the concept of a reflecting cardinal that allows us to force $\BSCFA(I)$, for adequate classes $I$ -- recall that for some classes, the resulting principle is inconsistent. The terminology adopted in the following definition is inspired by Bagaria \cite{Bagaria:C(n)-cardinals}. In that paper, for a natural number $n$, $C^{(n)}$ is defined to be the club class of ordinals $\alpha$ such that $\V_\alpha\prec_{\Sigma_n}\V$. It is pointed out there that if $\alpha\in C^{(1)}$, then $\alpha$ is an uncountable strong limit cardinal, and $\V_\alpha=H_\alpha$.

\begin{defn}
\label{defn:Sigma-n-reflectingCardinal}
An inaccessible cardinal \emph{$\kappa$ is $C^{(n)}$-reflecting} iff for every formula $\psi(x)$, every $a\in H_\kappa$, and every $C^{(n)}$-cardinal $\theta>\kappa$, if $\kla{H_\theta,\in}\models\psi(a)$, then there is a $C^{(n)}$-cardinal $\btheta<\kappa$ such that $a\in H_\btheta$ and $\kla{H_\btheta,\in}\models\psi(a)$.
\end{defn}

Note that an inaccessible cardinal above which there are no $\Cn$ cardinals is vacuously $\Cn$-reflecting. It would maybe have been more natural to require the existence of a proper class of $\Cn$ cardinals as part of the definition of $\Cn$-reflecting cardinals. In the application, this additional assumption will be made.

\begin{obs}
\label{obs:SmallForcingPreservesC(n)cardinals}
Let $\theta$ be a $C^{(n)}$-cardinal, let $\P\in H_\theta$ be a forcing notion, and let $G$ be $\P$-generic over $\V$. Then $\theta$ is a $C^{(n)}$-cardinal in $\V[G]$.
\end{obs}

\begin{proof}
It is well-known that inaccessible cardinals are preserved by small forcing. Thus, for $n=0$, nothing has to be shown, so let $n\ge 1$.
Let $\phi(x)$ be a $\Sigma_n$-formula and let $a\in H_\theta^{\V[G]}$. Noting that $H_\theta^{\V[G]}=H_\theta[G]$, it has to be shown that
\[\kla{H_\theta[G],\in}\models\phi(a)\iff\kla{\V[G],\in}\models\phi(a)\]
Let $a=\dot{a}^G$, where $\dot{a}\in H_\theta$.

Let's assume that $\kla{\V[G],\in}\models\phi(a)$. Let $p\in G$ be such that $p\forces_\P\phi(\dot{a})$. Let $\phi'(p,\P,\dot{a})$ be a formula that expresses that $p\forces_\P\phi(\dot{a})$. Since $\phi$ is $\Sigma_n$, it follows that $\phi'$ is also $\Sigma_n$. Indeed, consulting any standard text on forcing, one can see that there is a $\Delta_1$ function $F$ such that for every atomic formula $\chi$ in the forcing language for $\P$ and every $p\in\P$, $F(p,\chi)=1$ if $p\forces_\P\chi$ and $F(p,\chi)=0$ otherwise. This can be used to see that the formula $\phi'$ described before is $\Sigma_n$. Since $\kla{\V,\in}\models\phi'(p,\P,\dot{a})$ and $p,\P,\dot{a}\in H_\theta$, it follows that $\kla{H_\theta,\in}\models\phi'(p,\P,\dot{a})$, which means that $\kla{H_\theta,\in}\models(p\forces_\P\phi(\dot{a}))$, so that $\kla{H_\theta[G],\in}\models\phi(a)$. The converse is proven analogously.
\end{proof}

\begin{obs}
\label{obs:SmallForcingPreservesSigma-n-ReflectingCardinals}
Suppose $\P\in H_\kappa$ is a notion of forcing, where $\kappa$ is $C^{(n)}$-reflecting. Let $G\sub\P$ be generic. Then in $\V[G]$, $\kappa$ is still $C^{(n)}$-reflecting.
\end{obs}

\begin{proof}
Working in $\V[G]$, let $\psi(x)$, $a\in H_\kappa^{\V[G]}$ and $\theta$ be as in the Definition \ref{defn:Sigma-n-reflectingCardinal}. Let $\dot{a}\in H_\kappa^\V$ be such that $a=\dot{a}^G$. Let $\psi^*(q,\P,\dot{a})$ be the statement expressing, for $q\in\P$, that $q$ forces with respect to $\P$ that $\psi(\dot{a})$ holds. Clearly, there is a $p\in G$ such that $\kla{H_{\theta},\in}^\V\models\psi^*(p,\P,\dot{a})$. Since $\kappa$ is \Cn-reflecting in $\V$, there is then (in $\V$) a \Cn-cardinal $\btheta<\kappa$ such that $\kla{H_{\btheta},\in}^\V\models\psi(p,\P,\dot{a})$, and in particular, $p,\P,\dot{a}\in H_{\btheta}^\V$. It follows by Observation \ref{obs:SmallForcingPreservesC(n)cardinals} that $\btheta$ is a \Cn-cardinal in $\V[G]$, and since $H_{\btheta}^{\V[G]}=H_{\btheta}[G]$, it follows further that $\kla{H_{\btheta},\in}^{\V[G]}\models\psi(a)$.

\end{proof}

\begin{lem}
\label{lem:ForcingEnhancedBFAfromCnReflecting}
Let $I$ be a class term that's $\Delta_{n+1}$ (with respect to a fragment of $\ZFC$ that holds in every model of the form $H_\theta$, where $\theta$ is an uncountable cardinal), possibly involving parameters from $H_\kappa$, such that subcomplete-necessarily, $I\cap\HSC$ is immune to subcomplete forcing, meaning that for any subcomplete forcing $\P$, if $G$ is $\P$-generic over $\V$, then $I\cap\HSC=I^{\V[G]}\cap\HSC^\V$, and this remains true in every set-forcing extension by a subcomplete forcing. Suppose that $\kappa$ is a \Cn-reflecting cardinal, and suppose that there is a proper class of \Cn-cardinals. Then there is a subcomplete forcing $\P$ which forces $\BSCFA(I)$.
\end{lem}

\begin{proof}
Let $I$ be $\Delta_{n+1}$ in the parameters $\vec{b}\in H_\kappa$.
I will construct an RCS iteration of subcomplete forcing notions of length $\kappa$. Following the usual setup, this amounts to constructing sequences $\seq{\P_\alpha}{\alpha\le\kappa}$ and $\seq{\dot{\Q}_\alpha}{\alpha<\kappa}$ such that $\P_{\alpha+1}=\P_\alpha*\dot{\Q}_\alpha$, and $\P_\lambda$ is the RCS limit of the construction up to $\lambda$, for limit $\lambda$. So, assuming $\P_\alpha$ is defined, where $\alpha$ is an ordinal less than $\kappa$, it suffices to define $\dot{\Q}_\alpha$, and thus $\P_{\alpha+1}$. To this end, suppose $G_\alpha$ is $\P_\alpha$-generic over $\V$. Inductively, we will have that $\P_\alpha\in\V_\kappa$, and that for all $\beta<\alpha$, $\P_\beta$ is subcomplete and $\forces_{\P_\beta}$``$\dot{\Q}_\beta$ is subcomplete,'' and $\P_\beta$ has size $\omega_1$ in $\V^{\P_{\beta+1}}$.

In $\V[G_\alpha]$, let, for every $\Sigma_1$-formula $\phi=\phi(x)$ in the language of set theory with an additional unary predicate symbol $\dot{I}$, and for every $a\in H_{\omega_2}^{\V[G_\alpha]}$, $\theta_{\phi,a}$ be the least $\theta$ such that there is a subcomplete forcing $\P$ whose subcompleteness is verified by $\theta$, and which is such that the statement $\psi_\phi(\P,a,\vec{b})$ holds, expressing that $\kla{H_{\omega_2},\in,I\cap H_{\omega_2}}^{{\V[G_\alpha]}^\P}\models\phi(a)$, if there is such a $\P$.
If there is no such $\P$ at all, then let $\theta_{\phi,a}=0$.

Let $\Q_\alpha\in\V[G_\alpha]$ be the lottery sum of all subcomplete forcing notions whose subcompleteness is verified by $\sup_{\phi,a}\theta_{\phi,a}$, followed by $\Col{\omega_1,\P_\alpha}$, and let $\dot{\Q}_\alpha$ be a $\P_\alpha$-name of minimal rank such that $\eins_{\P_\alpha}$ forces that $\dot{\Q}_\alpha$ satisfies the definition given. Adopting terminology popularized by Hamkins, by the lottery sum of a collection of forcing notions, I mean the disjoint union of the posets in the collection, with a common weakest condition above all the conditions in the disjoint union. Thus, effectively, forcing with this sum amounts to choosing one of the posets in the collection and forcing with it. It was shown in \cite{Kaethe:Diss} that the lottery sum of a collection of subcomplete forcing notions is subcomplete.

\claim{(1)}{$\theta_{\phi,a}<\kappa$.}

\begin{proof}[Proof of (1)]

Work in $\V[G_\alpha]$, where $\kappa$ is \Cn-reflecting, by Observation \ref{obs:SmallForcingPreservesSigma-n-ReflectingCardinals}. Fix a $\Sigma_1$-formula $\phi(x)$ and a set $a\in\HSC^{\V[G_\alpha]}$ such that in $\V[G_\alpha]$, $\psi_\phi(\P,a,\vec{b})$ holds, for some subcomplete forcing $\P$ in $\V[G_\alpha]$ (if there is no such $\P$, then $\theta_{\phi,a}=0$, and there is nothing to show). Note that $\omega_2^{\V[G_\alpha]}<\kappa$, so that $a\in\V_\kappa^{\V[G_\alpha]}$. By assumption, $\vec{b}\in\V_\kappa\sub\V_\kappa^{\V[G_\alpha]}$ as well.

Let $\theta$ be a \Cn-cardinal such that in $\kla{H_\theta^{\V[G_\alpha]},\in}$, the following statement $\psi(a,\vec{b})$ holds: ``there are a forcing notion $\P'$, a regular cardinal $\tau$ and a set $\dot{A}$ such that $\power(H_\tau)$ exists, $\tau$ verifies the subcompleteness of $\P'$, $\dot{A}=\{\kla{\nu,q}\in H_\tau\st q\in\P'\ \text{and}\ q\forces_{\P'}\anf{\nu\in I}\}$, and
$\forces_{\P'}\kla{H_{\omega_2},\in,\dot{A}\cap H_{\omega_2}}\models\phi(\check{a})$.'' It is easy to see that such a $\theta$ exists, because, working in $W=\V[G_\alpha]$, if $\P'$ is a forcing notion whose subcompleteness is verified by some regular $\tau>\omega_2$, and is such that $\psi_\phi(\P',a,\vec{b})$ holds, then if $\theta'$ is any regular cardinal with $H_\tau\in H_\theta$, it will be true in $H_\theta$ that the subcompleteness of $\P'$ is verified by $\tau$ (see \cite[Lemma 2.1]{Jensen2014:SubcompleteAndLForcingSingapore}). Moreover, fixing $\P'$, if $\theta$ is in addition chosen to be a \Cn cardinal, then it follows that $I\cap H_\theta=I^{H_\theta}$, and similarly, $\dot{A}\cap H_\theta=\dot{A}^{H_\theta}$ (still working in $W$). This is because $\dot{A}$ can be defined in $\kla{H_\btheta,\in}$ both by a $\Sigma_{n+1}$- and a $\Pi_{n+1}$-formula. In more detail, let $\pi(x,\vec{b})$ be a $\Pi_{n+1}$ definition of $I$, and let $\sigma(x,\vec{b})$ be a $\Sigma_{n+1}$ definition of $I$. The equivalence between these formulas is assumed to be provable in a fragment of \ZFC that holds in every set of the form $H_\xi$, for any uncountable cardinal $\xi$. It follows then that when $H_\xi\prec_{\Sigma_n}W$, we have that $I^{\kla{H_\xi,\in}}=I^\V\cap H_\xi$, because if $c\in I^{\kla{H_\xi,\in}}$, then this means that $\kla{H_\xi,\in}\models\sigma(c,\vec{b})$, which implies that $\sigma(c,\vec{b})$ holds, because since $H_\xi\prec_{\Sigma_n}\V$, $\Sigma_{n+1}$-formulas go up. Vice versa, if $c\in I^\V\cap H_\xi$, then this means that $\pi(c,\vec{b})$ holds, and this implies that $\pi(c,\vec{b})$ holds in $\kla{H_\xi,\in}$ as well, since $\Pi_{n+1}$ formulas go down. The same argument can be carried with $\dot{A}$ in place of $I$, since ``$q\forces_{\P'}\nu\in I$'' can be expressed by ``$q\forces_{\P'}\sigma(\nu)$'' (which amounts to a $\Sigma_{n+1}$-formula) or ``$q\forces_{\P'}\pi(\nu)$'' (which amounts to a $\Pi_{n+1}$-formula).
Since there are arbitrarily large \Cn cardinals in $\V$, the same is true in $W$, and as a result, $\theta$ can be chosen so that  $\psi(a,\vec{b})$ holds in $\kla{H_\theta^{\V[G_\alpha]},\in}$.

Let $\btheta<\kappa$ be a \Cn-cardinal in $\V[G_\alpha]$ with $a,\vec{b}\in H_\btheta$ such that $\kla{H_\btheta,\in}\models\psi(a,\vec{b})$. Let $\bar{\P},\bar{\tau},\dot{\bar{A}}$ witness this.

It follows that $\btau$ really verifies the subcompleteness of $\bar{\P}$, and that
\[\forces_{\bar{\P}}\anf{\kla{H_{\omega_2},\in,I\cap H_{\omega_2}}\models\phi(\check{a})}.\]
Thus, $\theta_{\phi,a}\le\btau<\kappa$, as claimed.
\end{proof}

Since $\kappa$ is regular in $\V[G_\alpha]$, it follows that $\sup_{\phi,a}\theta_{\phi,a}<\kappa$, and hence that $\P_{\alpha+1}\in\V_\kappa$. This defines the iteration. Let $\P_\kappa$ be its RCS-limit.

Let $G$ be generic for $\P=\P_\kappa$. Standard arguments show that $\P$ is $\kappa$-cc, and as a consequence, it follows that $\kappa=\omega_2^{\V[G]}$. Let $\phi=\phi(a)$ be a $\Sigma_1$-formula in the parameter $a\in H_{\omega_2}^{\V[G]}$, in the language of set theory with an extra unary predicate symbol $\dot{I}$, and suppose that there is a subcomplete forcing $\Q=\dot{\Q}^G$ such that if $H$ is $\Q$-generic over $\V[G]$, then
$\kla{H_{\omega_2},\in,I\cap H_{\omega_2}}^{\V[G][H]}\models\phi(a)$. Let $a\in H_{\omega_2}^{\V[G_\alpha]}$, and let $p\in G$ force that $\dot{\Q}$ is as described.

Since $\kappa$ is still \Cn-reflecting in $\V[G_\alpha]$, we are in the same situation in $\V[G_\alpha]$ as we are in $\V$, so let's assume that $a\in H_{\omega_2}^\V$. There is then a subcomplete forcing that forces that in the structure $\kla{H_{\omega_2},\in,I\cap H_{\omega_2}}$ (in the sense of the forcing extension), $\phi(a)$ holds, namely, by the naturalness of the class of subcomplete forcing notions, there is a subcomplete forcing notion $\reals$ that's equivalent to $(\P_\kappa)_{{\le}p}*\dot{\Q}$. Pick a \Cn-cardinal $\theta>\kappa$ such that $\reals\in H_\theta$, and such that $H_\tau\in H_{\theta}$, where $\tau$ verifies the subcompleteness of $\reals$. Let $\psi(a,\vec{b})$ be the statement of the proof of (1). Then $\psi(a,\vec{b})$ holds in $\kla{H_\theta,\in}$. Now let $\btheta<\kappa$ be a \Cn-cardinal such that $a,\vec{b}\in H_\btheta$, and $\kla{H_\btheta,\in}\models\psi(a,\vec{b})$. Let $\btau,\bar{\P}$ witness this. Then $\bar{\P}$ is indeed subcomplete, and this is verified by $\btau$, and by the $\Sigma_n$-correctness of $\btheta$, it follows that $\bar{\P}$ actually forces that $\phi(a)$ holds in $\kla{H_{\omega_2},\in,\dot{I}\cap H_{\omega_2}}$. So there is a forcing notion in the lottery sum at stage $0$ of the iteration which will make $\phi(a)$ true in $\kla{\HSC,\in,\dot{I}\cap H_{\omega_2}}$, and hence it is dense that such a forcing notion was chosen. Once $\phi(a)$ is true in some $\kla{\HSC,\in,I\cap\HSC^{\V[G_\alpha]}}$, it persists to $\kla{\HSC,\in,I\cap\HSC^{\V[G]}}$, since $I\cap\HSC$ is subcomplete-necessarily immune to subcomplete forcing, and since $\phi$ is $\Sigma_1$. Thus, since $\phi$ and $a$ were chosen arbitrarily, we have shown that the condition stated in Definition \ref{defn:BFA(I)} holds in $\V[G]$.
\end{proof}


\end{document}